\documentclass[final,1p,times]{elsarticle}
\usepackage{lmodern}
\usepackage[T1]{fontenc}
\usepackage[utf8]{inputenc}
\usepackage[english]{babel}
\usepackage{booktabs}
\usepackage{amssymb}
\usepackage{amsmath}
\usepackage{amsthm}
\usepackage{mathtools}
\usepackage{float}
\usepackage{graphicx}
\usepackage{url}
\usepackage[colorlinks=true,linkcolor=blue,citecolor=blue,urlcolor=blue]{hyperref}
\journal{}
\usepackage{microtype}
\usepackage{enumitem}

\numberwithin{equation}{section}
\newtheorem{theorem}{Theorem}[section]

\newtheorem{corollary}[theorem]{Corollary}
\newtheorem{proposition}[theorem]{Proposition}
\newtheorem{hypothese}[theorem]{Assumption}
\newtheorem{remark}[theorem]{Remark}
\allowdisplaybreaks
\newcommand{\R}{\mathbb{R}}
\newcommand{\C}{\mathbb{C}}
\newcommand{\he}{^{\mathsf{H}}}
\newcommand{\tr}{^{\mathsf{T}}}
\newcommand{\Pe}{\mathrm{Pe}}
\newcommand{\Ah}{\widehat{A}}
\newcommand{\Kh}{\widehat{K}}
\newcommand{\Gh}{\widehat{G}}
\newcommand{\lmin}{\lambda_{\min}}
\newcommand{\lmax}{\lambda_{\max}}
\newcommand{\cond}{\kappa}
\newcommand{\norm}[1]{\left\lVert #1\right\rVert}
\newcommand{\ip}[2]{\left\langle #1,\,#2\right\rangle}
\newcommand{\Schur}[2]{#1\,\slash\,#2}
\DeclareMathOperator{\rang}{rank}

\DeclareMathOperator{\diag}{diag}
\DeclareMathOperator{\Div}{div}
\DeclareMathOperator{\Real}{Re}

\makeatletter
\def\ps@pprintTitle{%
	\let\@oddhead\@empty
	\let\@evenhead\@empty
	\let\@oddfoot\@empty
	\let\@evenfoot\@empty
}
\makeatother

\begin{document}

\begin{frontmatter}

\title{Double screening and Péclet threshold in VPINN training for heterogeneous coupled convection-diffusion-reaction systems}

\author{Ali Ouattara Kouma}
\ead{benzeus96@gmail.com}
\author{Gossrin Jean-Marc Bomisso}
\ead{bogojm@yahoo.fr}
\address{Universit\'e Nangui Abrogoua, UFR Sciences Fondamentales et Appliqu\'ees, 02 BP 801 Abidjan 02, C\^ote d'Ivoire}

\begin{abstract}
		We study the training of variational physics-informed neural networks (VPINNs) for linear coupled convection-diffusion-reaction systems in which only some components are transported. We ask what slows down training when convection dominates, from which Péclet number, and how to mitigate it. In the neural tangent kernel (NTK) regime, the residuals evolve under a Gram matrix. We prove a double screening theorem. When the convection strength tends to infinity, the Schur complement of the Gram matrix relative to the convective block converges to a limit from which the cross coupling blocks disappear. The kernel enters this limit only through its own Schur complement. An exact identity gives this Schur complement at every finite convection strength. A discrete version of the theorem applies to the Gram matrix of an actually constructed network. Under an explicit non-degeneracy condition, it yields a condition number growing like the square of the Péclet number and a Péclet threshold of the conditioning, whose asymptotic approximation is computed from two assemblies of the Jacobian. On a two-dimensional network, this approximation agrees with the exact threshold to within one percent, while omitting the second screening overestimates the threshold by up to sixteen percent. The threshold is defined spectrally, and its relevance for the cost of training depends on the optimization dynamics. Training experiments on steady problems show that the conditioning imposed by the shared step size is a major contributor to the degradation. A result on group-wise preconditioning shows that disjoint subnetworks admit a block rescaling of the step size that a shared trunk does not admit. It is consistent with the faster training of disjoint subnetworks observed under Adam.
	\end{abstract}

\begin{keyword}
physics-informed neural networks \sep neural tangent kernel \sep coupled convection-diffusion-reaction systems \sep dominant convection \sep Schur complement
\MSC[2020] 65M12 \sep 65K10 \sep 68T07 \sep 35K40 \sep 35K57
\end{keyword}

\end{frontmatter}

\section{Introduction}
	\label{sec:intro}

Physics-informed neural networks are widely used to approximate the solutions of partial differential equations \citep{raissi2019}. In their variational formulation, the network represents the unknown solution and its parameters minimize a loss built from the weak form of the residual over a family of test functions \citep{kharazmi2019,kharazmi2021}. These methods remain difficult to train, in particular when convection dominates diffusion.

Two lines of work address this difficulty. The first one concerns approximation, through a priori error bounds and generalization estimates \citep{shin2020,mishra2022,deryck2022,berrone2022,hu2023}. The second one concerns training dynamics. It relies on neural tangent kernel theory \citep{jacot2018}, which describes gradient flow in the limit of infinite width. For physics-informed networks, this theory has revealed a spectral bias and an imbalance between the terms of the loss \citep{wang2022,wang2021,krishnapriyan2021}. The remedies proposed there rescale the terms of the loss adaptively during training, from gradient statistics or from the trace of the kernel. These works deal mostly with scalar equations. For coupled systems, the spectral radius of the kernel grows with the coupling strength, and a block preconditioner of Gauss--Newton type removes this dependence \citep{park2026}. The coupling is treated there as a global scalar parameter. The nature of the coupled components does not enter the analysis. For VPINNs, a least squares solver for the weights of the last layer improves the convergence of gradient-based training \citep{uriarte2025}. For one-way coupled systems, in which each equation depends only on the outputs of the preceding ones, the difficulties of standard PINNs are analyzed and PINN models tailored to this structure are proposed in \citet{kamil2026}. The systems considered here are coupled in both directions through $H$, and the analysis concerns the variational formulation. For linear-constrained PINNs, including weak-form constraints, checkable rank criteria for the positivity of the limiting tangent kernel and a finite-width convergence result for gradient descent are given in \citet{lyu2026}. That work does not address the dependence of the spectrum on the convection strength. The discrete theorem of Section~\ref{ssec:seuil-discret} applies to the full Gram matrix of an actually constructed network.

We consider instead heterogeneous coupled systems. Some components undergo dominant convective transport, the others evolve mainly by diffusion, and a linear coupling connects the two groups. This structure appears in heat exchangers, in reactive transport in porous media and in coupled heat-moisture drying. For the finite element discretization of such systems, dominant convection calls for a stabilization of the convective components only \citep{johnnovo2011,roos2008}. For VPINNs, a stabilization term in the spirit of SUPG and learned boundary indicator functions improve the accuracy on convection-dominated scalar problems \citep{anandh2025}. This article addresses the following question. When convection dominates in a heterogeneous coupled system, what slows down the training of a variational network, from which Péclet number does the slowdown appear, and how can it be mitigated? We split it into three questions.
\begin{enumerate}[label=\textup{(Q\arabic*)},leftmargin=3em]
	\item What is the mechanism of the slowdown?
	\item From which Péclet number does it appear, and can this number be computed before training?
	\item Which remedies follow from the mechanism?
\end{enumerate}

The analysis starts from a reduction of the dynamics. In the neural tangent kernel regime, the residuals satisfy the linear system \[\dot r=-Gr,\] where $G$ is a Gram matrix. After frequency localization, this matrix reads, at each frequency, \[\Gh=\Ah\;\Kh\;\Ah\he,\] where $\Ah$ is the symbol of the system and $\Kh$ the matrix tangent kernel.

Our answers are the following.
\begin{enumerate}[label=\textup{(A\arabic*)},leftmargin=3em]
	\item The mechanism is a double screening (Theorem~\ref{thm:main}). Under dominant convection, the Schur complement of $\Gh$ relative to the convective block converges to $S_{DD}\,(\Schur{\Kh}{\Kh_{CC}})\,S_{DD}\he$. The cross coupling blocks and the part of the kernel correlated with the convective components disappear in the limit. The exchanges contained in the diffusive block remain. An exact identity valid at every finite convection strength fixes the order of convergence (Proposition~\ref{prop:exacte}). The largest eigenvalue grows like $b^{2}$ while the smallest one stays bounded, so that the condition number grows like $\Pe^{2}$ (Corollary~\ref{cor:cond}). For fully connected networks at infinite width, the kernel is diagonal and only the first screening is active (Proposition~\ref{prop:ntk-diag}).
	\item The threshold follows from the two ends of the spectrum. We distinguish the exact threshold, obtained by a scalar root search, from its closed-form asymptotic approximation. In the frequency model, the approximation reduces at infinite width to a formula involving only the symbol (Corollary~\ref{cor:ntk-seuil}). A discrete double screening theorem (Theorem~\ref{thm:discret}) gives a second form for an actually constructed network, valid under an explicit non-degeneracy condition and computed from two assemblies of the Jacobian. It is the main result for actually constructed networks. On a two-dimensional network, its second projection reduces the median error of the threshold approximation from $7.5\%$ to $0.06\%$.
	\item The conditioning imposed by the step size shared between equations controls the cost of training the diffusive block (Proposition~\ref{prop:iters}). A selective weighting of the residuals keeps the condition number bounded at the reference frequency, at the price of degrading the modes almost orthogonal to the transport in dimension two (Proposition~\ref{prop:weight}). Disjoint subnetworks allow a block rescaling of the step size, while a shared trunk requires slowing down the trunk relative to the other parameters (Proposition~\ref{prop:groupes}).
\end{enumerate}
The frequency analysis is an idealized model, and the training experiments concern steady problems. Three features distinguish this approach from the works cited above. The structure of the system, and not only the strength of the coupling, determines the limit spectrum. The threshold is computed without any training run. The remedies are fixed a priori by the partition into convective and diffusive components.

Section~\ref{sec:cadre} fixes the notation and presents the system, the network, the frequency reduction and the test systems. Sections~\ref{sec:main}, \ref{sec:seuil} and \ref{sec:optim} answer (Q1), (Q2) and (Q3) respectively. Each of them ends with the numerical validation of its answer.

\section{Framework}
\label{sec:cadre}

\subsection{Notation}
	\label{sec:notation}
	
	For a matrix $M\in\C^{p\times q}$ with $(p, q) \in (\mathbb{N}^{\ast})^2,\; M\tr$ denotes its transpose and
	$M\he$ its conjugate transpose. The identity matrix of order $p$ is denoted by $I_p$. For
	$M$ Hermitian, $\lmin(M)$ and $\lmax(M)$ denote respectively the smallest and the
	largest eigenvalue. We write $M\succ 0$ when $M$ is positive definite and
	$M_1\preceq M_2$ for the Loewner order. The spectral norm is denoted by $\norm{M}_2$, while the Frobenius norm is denoted by $\norm{M}_F$. For a Hermitian positive definite matrix $M$, the spectral condition number is defined by
	$$
	\cond(M)=\dfrac{\lambda_{\max}(M)}{\lambda_{\min}(M)} \cdotp
	$$
	
	The index set is partitioned as
	\begin{equation}
		\label{eq:partition}
		\{1,\dots,n\}=C\,\dot\cup\,D,\qquad |C|=m\ge 1,\qquad |D|=n-m\ge 1,
	\end{equation}
	where $C$ gathers the convective components and $D$ the diffusive components.
	The corresponding blocks of $M$ are denoted by $M_{CC}$, $M_{DD}$, $M_{CD}$ and
	$M_{DC}$. When $M_{CC}$ is invertible, the Schur complement is
	\begin{equation}
		\label{eq:schur-def}
		\Schur{M}{M_{CC}} =M_{DD}-M_{DC}\,M_{CC}^{-1}\,M_{CD}.
	\end{equation}
	For $M$ Hermitian with $M_{CC}\succ 0$, one has the variational
	characterization
	\begin{equation}
		\label{eq:schur-var}
		x_D\he\,(\Schur{M}{M_{CC}})\,x_D
		=\min_{x_C\in\C^{m}}
		\begin{pmatrix}x_C\\x_D\end{pmatrix}\he
		M\begin{pmatrix}x_C\\x_D\end{pmatrix},
	\end{equation}
	the minimum being attained at $x_C^{\star}=-M_{CC}^{-1}M_{CD}\,x_D$.
	
	The spatial domain $\Omega\subset\mathbb{R}^{d}$, with $d\geq 1$, is a bounded Lipschitz open set. We set
	
	$$
	Q=\Omega\times(0,T),
	$$
	
	and we denote by $z=(x,t)$ a point of $Q$. The inner product of $L^2(Q)^n$ is denoted by $\ip{\cdot}{\cdot}$.
	
	We denote by $\xi$ the spatial frequency, chosen in the direction of the transport field, and by $\omega$ the temporal frequency. The space-time symbol of the operator of the system is denoted by $\Ah(\omega,\xi)$, while $S(\omega,\xi)$ denotes its part independent of the convection strength. The matrix tangent kernel in frequency is denoted by $\Kh(\omega,\xi)$, and the corresponding Gram matrix by $\Gh(\omega,\xi)$. We also introduce the diagonal projector onto the convective components,
	
	$$
	P=\diag\bigl(\mathbf{1}_{i\in C}\bigr).
	$$
	
	Finally, the reference Péclet number associated with a spatial frequency $\xi>0$ is defined by
	\begin{equation}
		\label{eq:peclet-def}
		\Pe=\frac{b}{\varepsilon_D \, \xi},
		\qquad
		\varepsilon_D =\min_{i\in D}\varepsilon_i,
	\end{equation}
	where $b$ denotes the convection strength and $\varepsilon_i$ the diffusion coefficients. It is built on the smallest diffusivity of the non-convected components, and not on the diffusivity of the transported ones. This choice reflects the mechanism studied here. The degradation is governed by the ratio between the convective scale $b|\xi|$ and the diffusive block $S_{DD}$.

	\subsection{The coupled convection-diffusion-reaction system}
	\label{ssec:systeme}

	Let $u=(u_1,\dots,u_n)\tr$ and $E=\diag(\varepsilon_1,\dots,\varepsilon_n)$ with $\varepsilon_i>0$. The field $\beta\in\R^d$ is constant and unit, and $b>0$ denotes the convection strength. The matrix $H\in\R^{n\times n}$ gathers the reaction terms and the linear exchanges between components. We assume $H$ symmetric positive semidefinite. The operator of the system is defined by
	\begin{equation}
		\label{eq:op-L}
		\mathcal{L}\,u=-\Div\left(E\,\nabla u\right)+b\,(\beta\cdot\nabla)\,Pu+H\,u.
	\end{equation}
	Only the components of $C$ carry the transport term. Given $f\in L^{2}\bigl(0,T;L^{2}(\Omega)^{n}\bigr)$ and $u_0\in L^{2}(\Omega)^{n}$, the evolution problem reads
	\begin{equation}
		\label{eq:systeme}
		\begin{cases}
			\dfrac{\partial u}{\partial t}+\mathcal{L}\,u=f & \text{in }\Omega\times(0,T),\\[6pt]
			u=0 & \text{on }\partial\Omega\times(0,T),\\[4pt]
			u(\cdot,0)=u_0 & \text{in }\Omega.
		\end{cases}
	\end{equation}
	The steady problem associated with a source $f\in L^{2}(\Omega)^{n}$ is given by
	\begin{equation}
		\label{eq:systeme-stat}
		\begin{cases}
			\mathcal{L}\,u=f & \text{in }\Omega,\\[4pt]
			u=0 & \text{on }\partial\Omega.
		\end{cases}
	\end{equation}

	Since $\beta$ is constant, integration by parts gives $\int_\Omega(\beta\cdot\nabla v)\,v\,\mathrm{d}x=0$ for every $v\in H^{1}_{0}(\Omega)$. The bilinear form of $\mathcal{L}$ therefore satisfies, for all $v\in H^{1}_{0}(\Omega)^{n}$,
	\[
	\int_\Omega\Bigl(\sum_{i=1}^{n}\varepsilon_i\lvert\nabla v_i\rvert^{2}+b\,(\beta\cdot\nabla)(Pv)\cdot v+v\tr Hv\Bigr)\mathrm{d}x\ \ge\ \min_i\varepsilon_i\,\norm{\nabla v}_{L^{2}(\Omega)}^{2}.
	\]
	The form is continuous and coercive on $H^{1}_{0}(\Omega)^{n}$. By the Lax--Milgram theorem, problem \eqref{eq:systeme-stat} admits a unique weak solution $u\in H^{1}_{0}(\Omega)^{n}$. The Galerkin method of \citet[Section~7.1]{evans2010} applies to systems with diagonal principal part, and gives a unique weak solution of \eqref{eq:systeme}, with
	\[
	u\in L^{2}\bigl(0,T;H^{1}_{0}(\Omega)^{n}\bigr)\cap C\bigl([0,T];L^{2}(\Omega)^{n}\bigr),
	\qquad
	\dfrac{\partial u}{\partial t}\in L^{2}\bigl(0,T;H^{-1}(\Omega)^{n}\bigr).
	\]

	The reference system used in the experiments has $n=3$ components, $C=\{1\}$ and $D=\{2,3\}$. Written component by component, the steady problem \eqref{eq:systeme-stat} reads
	\begin{equation}
		\label{eq:sys3-explicite}
		\begin{cases}
			-\varepsilon_1\Delta u_1+b\,\beta\cdot\nabla u_1+\alpha\,(u_1-u_2)=f_1 & \text{in }\Omega,\\[4pt]
			-\varepsilon_2\Delta u_2+\alpha\,(u_2-u_1)+\gamma\,(u_2-u_3)=f_2 & \text{in }\Omega,\\[4pt]
			-\varepsilon_3\Delta u_3+\gamma\,(u_3-u_2)=f_3 & \text{in }\Omega,\\[4pt]
			u_1=u_2=u_3=0 & \text{on }\partial\Omega,
		\end{cases}
	\end{equation}
	with $\alpha,\gamma>0$. The component $u_1$ is a fluid subject to convection and diffusion. The component $u_2$ is a wall, which exchanges with the fluid and with an insulating layer. The component $u_3$ is the insulating layer, which evolves by diffusion and exchanges with the wall. The evolution version adds $\partial u_i/\partial t$ to each equation and the initial condition of \eqref{eq:systeme}. System \eqref{eq:sys3-explicite} corresponds to the exchange matrix $H$ of \eqref{eq:sys3}. This class of systems covers the models of multi-region heat exchangers, in which convected fluid components are coupled, through linear exchanges, to diffusive solid components. The matrix $H$ encodes the exchange graph and contributes to the zeroth-order coercivity of the diffusive components.

	The analysis of Sections~\ref{sec:main} to \ref{sec:optim} is formulated for the evolution problem \eqref{eq:systeme}, through the space-time symbol at the frequency $(\omega,\xi)$. The steady problem \eqref{eq:systeme-stat} corresponds to the slice $\omega=0$. All the training experiments concern steady problems, on $\Omega=(0,1)$ and $\Omega=(0,1)^{2}$. For the evolution problem, only the symbol computations at $\omega\neq0$ of Section~\ref{ssec:num-alg} are carried out. 

	The experiments on system \eqref{eq:sys3-explicite} use
	$E=\diag(10^{-2},10^{-2},2\times 10^{-2})$ and
	\begin{equation}
		\label{eq:sys3}
		C=\{1\},\qquad D=\{2,3\},\qquad
		H=\begin{pmatrix}\alpha & -\alpha & 0\\ -\alpha & \alpha+\gamma & -\gamma\\ 0 & -\gamma & \gamma\end{pmatrix},
	\end{equation}
	where $\alpha=1$, $\gamma=1/2$, $\xi=\pi$ and $\omega=0$ unless stated otherwise. The exchange matrix, given by \eqref{eq:sys3}, has rank $2$.

	The threshold experiments of Section~\ref{sec:seuil} also use a second family, with $n=4$ components, two of which are convected, with a star
	exchange topology. A hub component exchanges there with the two convected
	components and with an outer layer, so that the convected components
	communicate only through the hub. Components $1$ and $2$ are convected, component $3$ is the hub and component $4$ is the outer layer, so that $C=\{1,2\}$ and $D=\{3,4\}$. The exchange matrix is the Laplacian of the graph, given by
	\[
	H=\begin{pmatrix}a_1&0&-a_1&0\\0&a_2&-a_2&0\\-a_1&-a_2&a_1+a_2+a_3&-a_3\\0&0&-a_3&a_3\end{pmatrix},
	\]
	where $(a_1,a_2,a_3)$ are the exchange rates. System $A$ has diffusivities
	$(5\times10^{-3},2\times10^{-2},8\times10^{-3},4\times10^{-2})$ and exchange
	rates $(0.8,1.6,0.4)$. System $B$ has diffusivities
	$(5\times10^{-3},2\times10^{-2},4\times10^{-4},4\times10^{-2})$ and exchange
	rates $(0.05,0.10,0.02)$.

	Figure~\ref{fig:system} represents the instance used in part of the
	experiments. An advected fluid exchanges with a wall, which in turn exchanges with an
	insulating layer. Only the first component carries the transport term.
	
	\begin{figure}[H]
		\centering
		\includegraphics[width=0.8\linewidth]{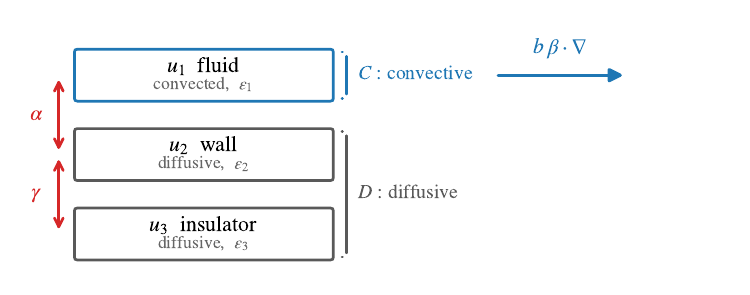}
		\caption{Reference system \eqref{eq:sys3}. An advected fluid is coupled through
			the exchange $\alpha$ to a wall, itself coupled through $\gamma$ to an
			insulating layer.}
		\label{fig:system}
	\end{figure}
	
	\subsection{Variational formulation and network}
	\label{ssec:vpinn}
	
	The network follows the neural tangent kernel parameterization. Let $L\ge 1$ be the
	number of hidden layers, $n_0$ the input dimension, equal to $d+1$ for the evolution problem and to $d$ for the steady problem, and $\sigma$ a
	continuously differentiable activation with bounded derivative, so that the Jacobian with respect to $\theta$ exists everywhere. The experiments use $\sigma=\tanh$. The preactivations are defined by
	\begin{equation}
		\label{eq:reseau}
		h^{(1)}(z)=\dfrac{1}{\sqrt{n_0}}\,W^{(1)}z+c^{(1)},
		\qquad
		h^{(\ell+1)}(z)=\dfrac{1}{\sqrt{n_\ell}}\,W^{(\ell+1)}\sigma\!\bigl(h^{(\ell)}(z)\bigr)+c^{(\ell+1)},
	\end{equation}
	and the vector output by
	\begin{equation}
		\label{eq:sortie}
		u_\theta(z)=\dfrac{1}{\sqrt{n_L}}\,W^{(L+1)}\sigma\!\bigl(h^{(L)}(z)\bigr)+c^{(L+1)}.
	\end{equation}
	The parameters $\theta$ gather the weights $W^{(\ell)}$ and the biases $c^{(\ell)}$, initialized according to the
	standard normal law. For one hidden layer, $M=n_1$ denotes the width. The factor $1/\sqrt{n_\ell}$ guarantees that at
	initialization the preactivations remain of order one, and that the tangent
	kernel admits a deterministic limit when the widths tend to
	infinity. For the steady problem \eqref{eq:systeme-stat}, the boundary condition is imposed through the multiplicative lifting
	\[
	u_\theta(x)=g(x)\,N_\theta(x),
	\]
	where $N_\theta$ is the output \eqref{eq:sortie} and $g\in W^{1,\infty}(\Omega)$ vanishes on $\partial\Omega$. For a Lipschitz activation, $N_\theta$ is Lipschitz on $\overline{\Omega}$, so that $u_\theta\in W^{1,\infty}(\Omega)^{n}$. The experiments use $g(x)=x(1-x)$ on $(0,1)$ and $g(x,y)=x(1-x)y(1-y)$ on $(0,1)^{2}$. The loss then contains only the interior residuals. The test functions have compact support in $Q$, so that the loss \eqref{eq:perte} does not contain the initial condition. For the evolution problem, the initial condition must be imposed by the lifting, for instance $u_\theta=u_0+t\,g\,N_\theta$ when $u_0\in H^{1}_0(\Omega)^{n}$. All the training experiments of this article concern the steady problem \eqref{eq:systeme-stat}.
	
	We denote by $\mathcal{A}=\partial_t+\mathcal{L}$ the space-time operator associated with \eqref{eq:systeme}, by $\mathcal{A}^{*}$ its formal adjoint, and by $e_i$ the $i$-th vector of the canonical basis of $\R^{n}$. Let $\{v_k\}_{k=1}^{N}\subset H^{2}(Q)$ be a family of test functions with compact support in $Q$, which is replaced by $\Omega$ for the steady problem \eqref{eq:systeme-stat}. Cubic B-splines with interior supports, and their tensor products, satisfy this condition, since they are of class $C^{2}$ and piecewise polynomial. It guarantees $\mathcal{A}^{*}(v_ke_i)\in L^{2}(Q)^{n}$, which is needed in \eqref{eq:gram}. The activation $\sigma$ is only of class $C^{1}$, so that $\mathcal{A}u_\theta$ need not belong to $L^{2}(Q)^{n}$. The variational residual is therefore defined by the weak form, given for the evolution problem by
	\begin{equation}
		\label{eq:residu}
				\begin{aligned}
			r_{k,i}(\theta)={}&-\int_Q u_{\theta,i}\,\dfrac{\partial v_k}{\partial t}\,\mathrm{d}z+\varepsilon_i\int_Q\nabla u_{\theta,i}\cdot\nabla v_k\,\mathrm{d}z+b\,\mathbf{1}_{i\in C}\int_Q(\beta\cdot\nabla u_{\theta,i})\,v_k\,\mathrm{d}z\\
			&+\int_Q(Hu_\theta)_i\,v_k\,\mathrm{d}z-\int_Qf_i\,v_k\,\mathrm{d}z.
		\end{aligned}
	\end{equation}
	It only requires $u_\theta\in H^{1}(Q)^{n}$, which holds for such an activation. For the steady problem, the time integral is removed and $Q$ is replaced by $\Omega$. When $u_\theta$ is smooth enough, $r_{k,i}(\theta)=\ip{\mathcal{A}u_\theta-f}{v_ke_i}=\ip{u_\theta}{\mathcal{A}^{*}(v_ke_i)}-\ip{f}{v_ke_i}$. The last expression is defined for every $u_\theta\in L^{2}(Q)^{n}$, and it is the one that leads to \eqref{eq:gram}. With $\sigma=\tanh$, the network is smooth and the three expressions coincide. The experiments compute $\ip{\mathcal{A}u_\theta-f}{v_ke_i}$, with the second derivatives of $u_\theta$ computed analytically, by a composite quadrature on a uniform grid. At the discrete level, this strong-form quadrature and a quadrature of the weak form \eqref{eq:residu} differ by a quadrature error. In the one-dimensional trainings, the gradient is computed from the adjoint expression of the residual, after integration by parts. Its relative gap to the gradient of the quadratured loss is at most $2.9\times10^{-5}$ on six configurations, measured by centered finite differences. In the two-dimensional experiments, the Jacobian is obtained by analytic differentiation of the assembled residual. The loss reads
	\begin{equation}
		\label{eq:perte}
		\mathcal{J}(\theta)=\dfrac{1}{2}\sum_{k=1}^{N}\sum_{i=1}^{n}\bigl|r_{k,i}(\theta)\bigr|^{2}
		=\dfrac{1}{2}\,\left\Vert r(\theta) \right\Vert_2^{2}.
	\end{equation}
	
	\subsection{Gradient flow and tangent kernel regime}
	\label{ssec:ntk}
	
	The neural tangent kernel is the matrix
	\begin{equation}
		\label{eq:ntk-def}
		\Theta_\theta(z,z')
		=\biggl[\ip{\partial_\theta (u_\theta)_p(z)}{\partial_\theta (u_\theta)_q(z')}\biggr]_{p,q=1}^{n}.
	\end{equation}
	Training by gradient flow reads $\dot\theta=-\nabla_\theta \mathcal{J}$. The
	chain rule gives, for the residual vector,
	\begin{equation}
		\label{eq:flot-r}
		\dot r=-\dfrac{\partial r}{\partial\theta}\left(\dfrac{\partial r}{\partial\theta}\right)\tr r
		=-G(\theta)\,r,
	\end{equation}
	where $G$ is the Gram matrix of the residual gradients. Denoting by
	$\mathcal{A}^{*}$ the formal adjoint of $\mathcal{A}$ and by $\mathcal{K}$ the integral
	operator with kernel $\Theta_\theta$, a direct computation yields
	\begin{equation}
		\label{eq:gram}
		G_{(k,i),(l,j)}=\ip{\mathcal{A}^{*}(v_k e_i)}{\mathcal{K}\,\mathcal{A}^{*}(v_l e_j)}.
	\end{equation}
	In the infinite width limit, $\Theta_\theta$ converges to a
	deterministic kernel $\Theta_\infty$ independent of $\theta$ and constant along
	training. System \eqref{eq:flot-r} then becomes linear with
	constant coefficients,
	\begin{equation}
		\label{eq:flot-lin}
		\dot r=-G\,r,\qquad r(t)= \exp \left( -G \, t \right)\,r(0),
	\end{equation}
	and the dynamics is encoded in the spectrum of $G$.
	
	\subsection{Frequency reduction}
	\label{ssec:symbole}
	For a fully
	connected network, two asymptotic results are established when the widths tend
	successively to infinity. At initialization, the output converges in law to
	a centered Gaussian process \citep[see][]{neal1996,lee2018}. The tangent kernel converges
	in probability to a deterministic kernel that remains constant under gradient
	flow \citep{jacot2018}. These results concern usual losses. Their extension to losses containing a differential operator was first established for one hidden layer \citep{wang2022}. It has since been developed for deep fully connected networks with losses informed by differential operators, for the convergence of the kernel at initialization and during training \citep{gan2025}. These results are formulated for strong-form residuals at collocation points and for a scalar output. The present setting uses variational residuals and a vector output with coupled components. We adopt the following assumption.
	
	\begin{hypothese}[\textbf{Tangent kernel regime}]
		\label{hyp:ntk}
		In the infinite width limit, the tangent kernel of the network
		\eqref{eq:reseau}-\eqref{eq:sortie} converges to a deterministic kernel
		$\Theta_\infty$, independent of $\theta$ and constant along training,
		so that the dynamics of the residuals is governed by \eqref{eq:flot-lin}.
	\end{hypothese}
	
	Assumption~\ref{hyp:ntk} provides a model of the training dynamics. Assumption~\ref{hyp:fourier} below is stronger. The limit kernel of a fully connected network with a smooth activation depends on the norms and on the inner product of its arguments, and is not translation invariant in general. A multiplicative lifting $u_\theta=g\,N_\theta$ multiplies the kernel by $g(z)g(z')$, which also breaks translation invariance \citep{xie2025}. Concentrated test functions do not diagonalize the finite Gram matrix exactly. The frequency analysis of Sections~\ref{sec:main} and \ref{ssec:seuil-freq} is therefore an idealized model, in the spirit of a von Neumann analysis. The results that apply to an actually constructed network rest on the discrete theorem of Section~\ref{ssec:seuil-discret}, which does not require Assumption~\ref{hyp:fourier}. The networks used in practice may depart from this regime, in particular when they are of finite width.

	The spectral analysis of $G$ reduces to that of a family of
	$n\times n$ matrices under the following assumption.
	
	\begin{hypothese}
		\label{hyp:fourier}
		The test functions have compact support in $Q$. The tangent kernel is
		translation invariant, that is,
		$$
		\Theta_\infty(z,z')=K(z-z'),
		$$
		and its Fourier transform $\Kh(\omega,\xi)$ is Hermitian positive
		definite. The test functions are concentrated in the neighborhood of a
		frequency $(\omega,\xi)$.
		
	\end{hypothese}
	
	Under Assumption~\ref{hyp:fourier}, the symbol of $\mathcal{A}$ at the frequency
	$(\omega,\xi)$ reads
	\begin{equation}
		\label{eq:symbole-A}
		\Ah(\omega,\xi)=\underbrace{E\,\xi^{2}+H+i\,\omega\,I}_{=:S(\omega,\xi)}+i\,b\,\xi\,P.
	\end{equation}
	The part $S$ is independent of $b$ and gathers the contributions of diffusion, of coupling and of the temporal frequency.
	The convective part is purely imaginary valued and carried by $P$. The Gram matrix is block diagonalized by frequency and reads
	\begin{equation}
		\label{eq:Ghat}
		\Gh(\omega,\xi)=\Ah(\omega,\xi)\,\Kh(\omega,\xi)\,\Ah(\omega,\xi)\he,
	\end{equation}
	up to a scalar mass factor. It is Hermitian positive semidefinite and becomes positive definite as soon as $\Ah$ is invertible, which we shall assume in what follows. In particular,
	$$
	\Ah_{DD}(\omega,\xi)=S_{DD}(\omega,\xi),
	$$
	and this block does not depend on $b$, since $P$ vanishes on the components of $D$.
	
	In dimension $d\geq2$, the same reduction applies mode by mode. For a Fourier mode $\mathbf{k}\in\R^d$, the symbol reads
	$$
	S=E\lvert\mathbf{k}\rvert^2+H+i\omega I,
	$$
	while the convective part is given by
	$
	i\,b\,(\beta\cdot\mathbf{k})\,P.
	$
	The results that follow then apply upon replacing $b\xi$ by
	$b\,(\beta\cdot\mathbf{k})$. The spectral quantities thus depend on
	the orientation of $\mathbf{k}$ relative to the transport direction $\beta$.
	
		\paragraph{Computational frameworks}
	Four computational frameworks are used. The exact symbol is evaluated in double
	precision. A discrete framework on $(0,1)$ uses $N=16$ cubic B-spline test
	functions per component, with interior supports, the Gaussian kernel $\Theta(z,z')=\exp\bigl(-\lvert z-z'\rvert^{2}/(2\ell_K^{2})\bigr)I_n$ with $\ell_K=0.05$ and a quadrature with $2001$ points, the matrix
	$G\in\R^{48\times 48}$ being assembled without training. An analogous framework on
	$(0,1)^{2}$ uses tensor products of B-splines and a quadrature on
	a $41\times 41$ grid. Finite width networks are finally trained
	completely.

	\section{The mechanism: double screening}
	\label{sec:main}

\subsection{The double screening theorem}
\label{ssec:thm}
	
	This section answers (Q1). The statement is purely algebraic. The matrix $S$ is an arbitrary complex matrix, which covers the steady case and the parabolic case.
	
	\begin{theorem}[\textbf{Double screening}]
		\label{thm:main}
		Let $\Gh=\Ah\;\Kh\;\Ah\he$ with $\Ah=S+ib\xi P$, where $S\in\C^{n\times n}$,
		$\xi\in\R\setminus\{0\}$, $P$ is the orthogonal projector onto the components of $C$, and
		$\Kh$ is Hermitian positive definite. Then, for $b$ large enough,
		$\Gh_{CC}$ is invertible and
		\begin{equation}
			\label{eq:main}
			\Schur{\Gh}{\Gh_{CC}}
			=
			S_{DD}\,\bigl(\Schur{\Kh}{\Kh_{CC}}\bigr)\,S_{DD}\he
			+O(b^{-1}).
		\end{equation}
	\end{theorem}
	
	We denote the limit by
	\[
	L_\infty =S_{DD}\,\bigl(\Schur{\Kh}{\Kh_{CC}}\bigr)\,S_{DD}\he.
	\]
	
	\begin{remark}
		The limit is independent of the cross blocks $S_{DC}$ and $S_{CD}$, which couple the convective and the diffusive components. The exchanges internal to the diffusive block, and the zeroth-order contributions of the exchanges with the convective components to the diagonal of $H_{DD}$, remain in $S_{DD}$. For system \eqref{eq:sys3} with $\Kh=I$, the smallest eigenvalue of the limit equals $0.0212$, $0.2256$ and $0.3294$ for $\alpha=0$, $1$ and $2$. The dependence on the blocks of the kernel is entirely carried by the Schur complement $\Schur{\Kh}{\Kh_{CC}}$.
	\end{remark}
	
	\begin{proof}
		The proof applies twice the characterization \eqref{eq:schur-var}.
		
		\emph{Step 1.}
		For $x=(x_C,x_D)$, let us set 
		\[
		q(x)=x\he\Gh x=\norm{\Kh^{1/2}\Ah\he x}_2^{2}
		\quad \text{ and } \quad
		y=\Ah\he x.
		\]
		Denoting $\widetilde{S}=S\he$, the block decomposition gives
		\begin{align}
			\label{eq:y-blocs}
			\begin{split}
				y_C & =\widetilde{S}_{CC}x_C+\widetilde{S}_{CD}x_D-ib\xi x_C,\\
				y_D & =\widetilde{S}_{DC}x_C+\widetilde{S}_{DD}x_D.
			\end{split}
		\end{align}
		For $b|\xi|>\norm{\widetilde{S}_{CC}}_2$, the matrix
		\[
		T_b=\widetilde{S}_{CC}-ib\xi I_m
		\]
		is invertible and
		\[
		\norm{T_b^{-1}}_2
		\leq
		\frac{1}{b|\xi|-\norm{\widetilde{S}_{CC}}_2}
		=O(b^{-1}).
		\]
		At fixed $x_D$, the map $x_C\mapsto y_C$ is therefore an affine bijection
		of $\C^m$.
		
		\emph{Step 2.}
		By \eqref{eq:schur-var} applied to $\Gh$,
		\begin{equation*}
			x_D\he\,(\Schur{\Gh}{\Gh_{CC}})\,x_D
			=
			\min_{x_C} \, q(x_C,x_D)
			=
			\min_{y_C}y\he\Kh y,
		\end{equation*}
		according to Step~1. Moreover,
		\begin{equation}
			\label{eq:yD-pert}
			y_D=\widetilde{S}_{DD}x_D+\delta(y_C,x_D),
			\qquad
			\delta =\widetilde{S}_{DC}T_b^{-1}
			\bigl(y_C-\widetilde{S}_{CD}x_D\bigr),
		\end{equation}
		with
		\[
		\norm{\delta}_2
		\leq
		\frac{c_1(\norm{y_C}_2+\norm{x_D}_2)}
		{b|\xi|-\norm{\widetilde{S}_{CC}}_2},
		\qquad
		c_1 =\max\bigl(
		\norm{\widetilde{S}_{DC}}_2,
		\norm{\widetilde{S}_{DC}}_2\norm{\widetilde{S}_{CD}}_2
		\bigr).
		\]
		The variable $y_C$ is free and, at fixed $y_C$, we have $x_C=O(b^{-1})$. Consequently,
		the contribution of $x_C$ to $y_D$ is of order $O(b^{-1})$. This is
		the screening of the convective block.
		
		\emph{Step 3.}
		Assume first that $\delta=0$. Then $y_D=\widetilde{S}_{DD}x_D$ is fixed, while
		$y_C$ remains free. By \eqref{eq:schur-var} applied to $\Kh$,
		\begin{equation*}
			\min_{y_C}
			\begin{pmatrix}y_C\\y_D\end{pmatrix}\he
			\Kh
			\begin{pmatrix}y_C\\y_D\end{pmatrix}
			=
			x_D\he\,\widetilde{S}_{DD}\he\,
			\bigl(\Schur{\Kh}{\Kh_{CC}}\bigr)\,
			\widetilde{S}_{DD}\,x_D,
		\end{equation*}
		the minimum being attained at
		\[
		y_C^{\star}=-\Kh_{CC}^{-1}\Kh_{CD}y_D,
		\]
		which satisfies
		\[
		\norm{y_C^{\star}}_2
		\leq
		c_2\norm{x_D}_2,
		\qquad
		c_2 =\norm{\Kh_{CC}^{-1}\Kh_{CD}}_2
		\norm{\widetilde{S}_{DD}}_2,
		\]
		a bound uniform in $b$. Since $\widetilde{S}_{DD}=(S_{DD})\he$, this value corresponds to the right-hand side
		of \eqref{eq:main} evaluated at $x_D$. Hence, in the limit, the kernel
		enters only through its Schur complement $\Schur{\Kh}{\Kh_{CC}}$.
		This is the screening of the kernel.
		
		\emph{Step 4.}
		It remains to show that the perturbation $\delta$ modifies the value of the
		minimum only at order $O(b^{-1})$, uniformly on the sphere
		$\norm{x_D}_2=1$.
		Let us denote by $F(y_C)$ the form
	$(y_C,\widetilde{S}_{DD}x_D+\delta(y_C,x_D))\he\Kh(\cdot)$ and by $F_0$ the unperturbed
	form obtained for $\delta\equiv 0$.
	
	Evaluating $F$ at the unperturbed minimizer
	$y_C^{\star}=-\Kh_{CC}^{-1}\Kh_{CD}\widetilde{S}_{DD}x_D$,
	\begin{equation*}
		\min_{y_C}F\le F(y_C^{\star})
		=F_0(y_C^{\star})
		+2\Real\ip{\Kh(y_C^{\star},\widetilde{S}_{DD}x_D)}{(0,\delta^{\star})}
		+\delta^{\star\mathsf{H}}\Kh_{DD}\delta^{\star},
	\end{equation*}
	where $\delta^{\star}=\delta(y_C^{\star},x_D)$ satisfies
	$\norm{\delta^{\star}}_2\le c_1(1+c_2)/(b|\xi|-\norm{\widetilde{S}_{CC}}_2)=O(b^{-1})$.
	The two correction terms are of order $O(b^{-1})$ and $O(b^{-2})$, whence
	$\min F\le\min F_0+O(b^{-1})$.
	
	Let now $y_C^{b}$ be a minimizer of $F$. Since $F(y_C^{b})\le
	F(y_C^{\star})$ is bounded uniformly in $b$ and
	$F(y_C)\ge\lmin(\Kh)\norm{y_C}_2^{2}$, the sequence $\norm{y_C^{b}}_2$ is bounded.
	Then $\norm{\delta(y_C^{b},x_D)}_2=O(b^{-1})$ uniformly, and the same
	expansion gives
	\begin{equation*}
		F(y_C^{b})\ge F_0(y_C^{b})-O(b^{-1})\ge\min_{y_C}F_0-O(b^{-1}).
	\end{equation*}
	Combining the two bounds, $\min F=\min F_0+O(b^{-1})$, uniformly on the
	sphere $\norm{x_D}_2=1$.
				Two Hermitian quadratic forms whose values agree up to
		$O(b^{-1})$ on the unit sphere have matrices at distance
		$O(b^{-1})$ in operator norm, whence \eqref{eq:main}. Finally,
		\[
		\Gh_{CC}=b^{2}\xi^{2}\Kh_{CC}+O(b)\succ0
		\]
		for $b$ large, so that the Schur complement is well defined.
	\end{proof}

	The variational proof does not require $S_{DD}$ to be invertible. When it is, the Schur complement admits a closed form valid at every finite $b$, which makes the rate of convergence explicit.

\begin{proposition}[\textbf{Exact screening identity}]
	\label{prop:exacte}
	Under the assumptions of Theorem~\ref{thm:main}, assume that $\Ah$ is invertible and that $b|\xi|>\norm{S_{CC}}_2$. Set
	\begin{equation}
		\label{eq:def-UY}
		U_b=S_{CC}+ib\xi I_m,
		\qquad
		\Sigma_b=S_{DD}-S_{DC}\,U_b^{-1}S_{CD},
		\qquad
		Y_b=\begin{pmatrix}-U_b^{-1}S_{CD}\\ I_{n-m}\end{pmatrix}.
	\end{equation}
	Then $\Sigma_b$ is invertible and
	\begin{equation}
		\label{eq:exacte}
		\Schur{\Gh}{\Gh_{CC}}
		=
		\Sigma_b\,\bigl(Y_b\he\,\Kh^{-1}\,Y_b\bigr)^{-1}\,\Sigma_b\he.
	\end{equation}
\end{proposition}

\begin{proof}
	The matrix $\Gh=\Ah\Kh\Ah\he$ is invertible, and the inverse of a Schur complement is the corresponding block of the inverse,
	\[
	\bigl(\Schur{\Gh}{\Gh_{CC}}\bigr)^{-1}=\bigl(\Gh^{-1}\bigr)_{DD},
	\qquad
	\Gh^{-1}=\Ah^{-\mathsf{H}}\,\Kh^{-1}\,\Ah^{-1}.
	\]
	The block $U_b=\Ah_{CC}$ is invertible and $\Sigma_b=\Schur{\Ah}{U_b}$. Since $\Ah$ is invertible, so is $\Sigma_b$, and the block inversion formula gives the columns $D$ of $\Ah^{-1}$ as $Y_b\Sigma_b^{-1}$. Hence
	\[
	\bigl(\Gh^{-1}\bigr)_{DD}=\Sigma_b^{-\mathsf{H}}\,Y_b\he\Kh^{-1}Y_b\,\Sigma_b^{-1},
	\]
	and inverting this identity gives \eqref{eq:exacte}.
\end{proof}

When $b\to\infty$, one has $U_b^{-1}\to0$, $\Sigma_b\to S_{DD}$ and $Y_b\he\Kh^{-1}Y_b\to(\Kh^{-1})_{DD}=(\Schur{\Kh}{\Kh_{CC}})^{-1}$, which recovers \eqref{eq:main}. The next term of the expansion distinguishes the real case from the complex case.

\begin{corollary}[\textbf{Order of convergence}]
	\label{cor:ordre}
	Under the assumptions of Proposition~\ref{prop:exacte} with $S_{DD}$ invertible, write $\Kh_\sharp=\Schur{\Kh}{\Kh_{CC}}$ and
	\begin{equation}
		\label{eq:terme-M}
		\begin{aligned}
		\mathsf{N}&=(\Kh^{-1})_{DC}S_{CD}-S_{CD}\he(\Kh^{-1})_{CD},
		\\
		\mathsf{M}&=S_{DC}S_{CD}\Kh_\sharp S_{DD}\he-S_{DD}\Kh_\sharp(S_{DC}S_{CD})\he-S_{DD}\Kh_\sharp\mathsf{N}\Kh_\sharp S_{DD}\he.
		\end{aligned}
	\end{equation}
	Then
	\begin{equation}
		\label{eq:dvlpt}
		\Schur{\Gh}{\Gh_{CC}}=L_\infty+\dfrac{i}{b\xi}\,\mathsf{M}+O(b^{-2}).
	\end{equation}
	If $S$ and $\Kh$ are real, then $\mathsf{M}$ is real and antisymmetric, and every simple eigenvalue of $L_\infty$ is approximated with an error $O(b^{-2})$. In general, the error on the eigenvalues is of order $O(b^{-1})$.
\end{corollary}

\begin{proof}
	The Neumann series gives $U_b^{-1}=-i(b\xi)^{-1}I_m+O(b^{-2})$. Hence
	\[
	\Sigma_b=S_{DD}+\dfrac{i}{b\xi}S_{DC}S_{CD}+O(b^{-2}),
	\qquad
	Y_b\he\Kh^{-1}Y_b=\Kh_\sharp^{-1}+\dfrac{i}{b\xi}\mathsf{N}+O(b^{-2}),
	\]
	and $(Y_b\he\Kh^{-1}Y_b)^{-1}=\Kh_\sharp-\frac{i}{b\xi}\Kh_\sharp\mathsf{N}\Kh_\sharp+O(b^{-2})$. Inserting these expansions in \eqref{eq:exacte} gives \eqref{eq:dvlpt}. When $S$ and $\Kh$ are real, $\mathsf{M}$ is real. Since the left-hand side of \eqref{eq:dvlpt} and $L_\infty$ are Hermitian, the matrix $i\mathsf{M}$ is Hermitian, so that $\mathsf{M}\tr=-\mathsf{M}$. For a simple eigenvalue $\lambda$ of the real symmetric matrix $L_\infty$, with real unit eigenvector $v$, first-order perturbation theory gives the correction $(b\xi)^{-1}\,i\,v\tr\mathsf{M}v=0$. The error is therefore $O(b^{-2})$. When $S$ is complex, for instance when $\omega\neq0$, the quantity $v\he(i\mathsf{M})v$ is in general nonzero.
\end{proof}

\begin{remark}
	\label{rem:ordre}
	For a multiple eigenvalue of $L_\infty$, the first-order correction is given by the eigenvalues of the compression of $i\mathsf{M}$ to the eigenspace. These eigenvalues come in pairs $\pm\mu$ and need not vanish. The convergence in operator norm is of order $O(b^{-1})$ in all cases.
 For $S=\bigl(\begin{smallmatrix}0&0&1\\1&1&0\\0&0&1\end{smallmatrix}\bigr)$, $\Kh=I_3$, $C=\{1\}$ and $\xi=1$, one has $L_\infty=I_2$, and the eigenvalues of the Schur complement are exactly $1\pm(1+b^{2})^{-1/2}=1\pm b^{-1}+O(b^{-3})$, although $S$ and $\Kh$ are real.
\end{remark}
	
	The behavior of both ends of the spectrum follows.
	
	\begin{corollary}[\textbf{Spectral asymptotics of} $\mathbf{\Gh}$]
		\label{cor:spectre}
		Under the assumptions of Theorem~\ref{thm:main}, with $S_{DD}$ invertible,
		\begin{align}
			\lmax(\Gh)&=b^{2}\xi^{2}\,\lmax\bigl(\Kh_{CC}\bigr)\,\bigl(1+O(b^{-1})\bigr),
			\label{eq:lmax}\\
			\lmin(\Gh)&\xrightarrow[b\to\infty]{}
			\lmin\Bigl(S_{DD}\,\bigl(\Schur{\Kh}{\Kh_{CC}}\bigr)\,S_{DD}\he\Bigr)>0.
			\label{eq:lmin}
		\end{align}
		Moreover, every unit eigenvector $v^{b}=(v_C^{b},v_D^{b})$ of $\Gh$ associated with $\lmin(\Gh)$ satisfies
		\[
		\norm{v_C^{b}}_2=O(b^{-1}),
		\]
		so that the slow mode localizes asymptotically in the diffusive block.
	\end{corollary}
	
	\begin{proof}
		For \eqref{eq:lmax}, we write
		\[
		\Gh=b^{2}\xi^{2}P\Kh P+bR_1+R_0,
		\]
		where
		\[
		R_1=i\xi(P\Kh S\he-S\Kh P),
		\qquad
		R_0=S\Kh S\he
		\]
		are bounded. Weyl's theorem then gives
		\[
		\bigl|\lmax(\Gh)-b^{2}\xi^{2}\lmax(P\Kh P)\bigr|
		\leq
		b\norm{R_1}_2+\norm{R_0}_2.
		\]
		Now,
		\[
		\lmax(P\Kh P)=\lmax(\Kh_{CC}),
		\]
		which gives \eqref{eq:lmax}.
		
		For \eqref{eq:lmin}, let us set
		\[
		L_b =\Schur{\Gh}{\Gh_{CC}},
		\qquad
		L_\infty =
		S_{DD}\bigl(\Schur{\Kh}{\Kh_{CC}}\bigr)S_{DD}\he.
		\]
		By Theorem~\ref{thm:main}, $L_b\to L_\infty$ and $L_\infty\succ0$.
		Applying \eqref{eq:schur-var} at the minimizer, we obtain
		\[
		\lmin(\Gh)\leq\lmin(L_b).
		\]
		
		Conversely, for every unit $x$,
		\[
		x\he\Gh x
		\geq
		\lmin(L_b)\norm{x_D}_2^{2},
		\qquad
		x\he\Gh x
		\geq
		\lmin(\Kh)\norm{(\Ah\he x)_C}_2^{2}.
		\]
		By \eqref{eq:y-blocs}, if
		$b|\xi|\geq2\norm{\widetilde{S}_{CC}}_2$, then
		\begin{equation*}
			\norm{(\Ah\he x)_C}_2
			\geq
			\tfrac12 b|\xi|\norm{x_C}_2-\norm{\widetilde{S}_{CD}}_2.
		\end{equation*}
		
		Let us now take $x$ to be a unit eigenvector associated with $\lmin(\Gh)$.
		Since
		\[
		\lmin(\Gh)=x\he\Gh x\leq\lmin(L_b),
		\]
		and $\lmin(L_b)$ is bounded in $b$, the second inequality above
		implies
		\[
		\norm{x_C}_2=O(b^{-1}).
		\]
		Consequently,
		\[
		\norm{x_D}_2^2=1-O(b^{-2}),
		\]
		and the first inequality gives
		\[
		\lmin(\Gh)
		\geq
		\lmin(L_b)\bigl(1-O(b^{-2})\bigr).
		\]
		This completes the proof of \eqref{eq:lmin}.
	\end{proof}
	
	\subsection{Consequences of the mechanism}
	\label{sec:coro}

	\subsubsection{Screening of the coupling}
	\label{ssec:coro-couplage}
	
	The raw diffusive block of the Gram matrix is
	$\Gh_{DD}=S_{D\bullet}\Kh S_{D\bullet}\he$, where $S_{D\bullet}$ denotes the rows
	$D$ of $S$. It does not depend on $b$ and contains the contributions of the
	coupling blocks. The following corollary describes the limit form obtained after screening.
	
	\begin{corollary}[\textbf{Coupling gap identity}]
		\label{cor:couplage}
		Under the assumptions of Theorem~\ref{thm:main},
		\begin{equation}
			\label{eq:ecart}
			\Gh_{DD}-L_\infty=X\,\Kh_{CC}\,X\he\succeq 0,
			\quad \text{ with } \quad 
			X=S_{DC}+S_{DD}\,\Kh_{DC}\,\Kh_{CC}^{-1}.
		\end{equation}
		In particular $L_\infty\preceq\Gh_{DD}$.
	\end{corollary}
	
	The gap between the raw diffusive block and its screening limit is therefore entirely determined by the matrix $X$. It depends jointly on the coupling of the system and on the coupling of the kernel.
	
	\begin{proof}
		Expanding $\Gh_{DD}$ gives
		\begin{equation*}
			\Gh_{DD}
			=
			S_{DC}\Kh_{CC}S_{DC}\he
			+
			S_{DC}\Kh_{CD}S_{DD}\he
			+
			S_{DD}\Kh_{DC}S_{DC}\he
			+
			S_{DD}\Kh_{DD}S_{DD}\he,
		\end{equation*}
		while
		\[
		L_\infty
		=
		S_{DD}\Kh_{DD}S_{DD}\he
		-
		S_{DD}\Kh_{DC}\Kh_{CC}^{-1}\Kh_{CD}S_{DD}\he.
		\]
		Using $\Kh_{CD}=\Kh_{DC}\he$, we obtain
		\[
		\Gh_{DD}-L_\infty
		=
		\bigl(S_{DC}+S_{DD}\Kh_{DC}\Kh_{CC}^{-1}\bigr)
		\Kh_{CC}
		\bigl(S_{DC}+S_{DD}\Kh_{DC}\Kh_{CC}^{-1}\bigr)\he
		=
		X\Kh_{CC}X\he.
		\]
		Positivity follows from $\Kh_{CC}\succ0$.
	\end{proof}
	
	The part of the coercivity that the exchanges with the convective components bring through the cross blocks therefore does not contribute to the effective training rate of the diffusive components. It is screened by the convective components, and \eqref{eq:ecart} quantifies this reduction precisely. The part contained in $S_{DD}$ remains. For system \eqref{eq:sys3-explicite}, the exchange coefficient $\alpha$ enters the diagonal of $H_{DD}$, and the limit rate depends on it, as the remark following Theorem~\ref{thm:main} shows.
	
	\subsubsection{Screening of the kernel}
	\label{ssec:coro-rho}
	
	For a finite width network whose convective and diffusive outputs share parameters, the empirical tangent
	kernel has nonzero cross blocks. The cost of this correlation is measured by
	the canonical correlations between blocks.
	
	\begin{corollary}[\textbf{Canonical correlation law}]
		\label{cor:rho}
		Let us set
		\[
		R =\Kh_{CC}^{-1/2}\Kh_{CD}\Kh_{DD}^{-1/2},
		\qquad
		\rho =\norm{R}_2.
		\]
		Then
		\begin{equation}
			\label{eq:schur-canonique}
			\Schur{\Kh}{\Kh_{CC}}
			=
			\Kh_{DD}^{1/2}
			\bigl(I-R\he R\bigr)
			\Kh_{DD}^{1/2},
			\qquad
			(1-\rho^2)\Kh_{DD}
			\preceq
			\Schur{\Kh}{\Kh_{CC}}
			\preceq
			\Kh_{DD},
		\end{equation}
		and, for the limit rate,
		\begin{equation}
			\label{eq:cor-rho}
			(1-\rho^2)\,
			\lmin\bigl(S_{DD}\Kh_{DD}S_{DD}\he\bigr)
			\leq
			\lmin(L_\infty)
			\leq
			\lmin\bigl(S_{DD}\Kh_{DD}S_{DD}\he\bigr).
		\end{equation}
		The upper bound is attained when $\Kh_{CD}=0$. Conversely, if it is attained, then $\Kh_{CD}S_{DD}\he w=0$ for every eigenvector $w$ of $S_{DD}\Kh_{DD}S_{DD}\he$ associated with its smallest eigenvalue.
	\end{corollary}
	
	\begin{proof}
		Since $\Kh\succ0$, the Schur criterion gives
		$I-R\he R\succ0$, hence $\rho<1$. The identity follows from a direct computation.
		The two-sided bound follows from
		\[
		(1-\rho^2)I\preceq I-R\he R\preceq I,
		\]
		from the congruences by $\Kh_{DD}^{1/2}$ and $S_{DD}$, which preserve the Loewner order, and from the monotonicity of $\lmin$. For the necessary condition, write $\Lambda_0=\lmin(S_{DD}\Kh_{DD}S_{DD}\he)$ and let $w$ be a unit eigenvector associated with $\Lambda_0$. Then
	\[
	w\he L_\infty w=\Lambda_0-\norm{\Kh_{CC}^{-1/2}\Kh_{CD}S_{DD}\he w}_2^{2}.
	\]
	If $\lmin(L_\infty)=\Lambda_0$, the left-hand side is at least $\Lambda_0$, so that $\Kh_{CD}S_{DD}\he w=0$.
	\end{proof}
	
	The correlation between the convective and diffusive heads can therefore lower the limit rate $\lmin(L_\infty)$, and the lowered rate remains at least $1-\rho^{2}$ times its value $\lmin(S_{DD}\Kh_{DD}S_{DD}\he)$ for an uncorrelated kernel. The condition $\Kh_{CD}=0$ is sufficient for the upper bound but not necessary. For $S_{DD}=\diag(1,2)$, $\Kh_{DD}=I_2$, $\Kh_{CC}=1$ and $\Kh_{CD}=(0,\ 0.6)$, one has $\rho=0.6$ and $\lmin(L_\infty)=\lmin(S_{DD}\Kh_{DD}S_{DD}\he)=1$. The correlation then acts only on the fast diffusive direction.

\subsubsection{Structure of the limit kernel}
\label{ssec:coro-ntk}

The results above hold for any Hermitian positive definite kernel. For the fully connected network \eqref{eq:reseau}-\eqref{eq:sortie}, the limit kernel has a particular structure. Its output components are independent at infinite width.

\begin{proposition}[\textbf{Diagonal limit kernel}]
	\label{prop:ntk-diag}
	For the network \eqref{eq:reseau}-\eqref{eq:sortie} with weights and biases initialized according to the standard normal law, the limit kernel reads
	\begin{equation}
		\label{eq:ntk-diag}
		\Theta_\infty(z,z')=\vartheta_\infty(z,z')\,I_n,
	\end{equation}
	where $\vartheta_\infty$ is a scalar kernel independent of $n$. Under Assumption~\ref{hyp:fourier}, this gives $\Kh(\omega,\xi)=\widehat{\vartheta}(\omega,\xi)\,I_n$ with $\widehat{\vartheta}(\omega,\xi)>0$.
\end{proposition}

\begin{proof}
	This is the limit theorem of \citet{jacot2018}, in which the limit tangent kernel is the tensor product of a scalar kernel with the identity of the output space. The contribution of the last layer to the entry $(p,q)$ is $\frac{1}{n_L}\sum_{j}\sigma(h^{(L)}_j(z))\sigma(h^{(L)}_j(z'))\,\delta_{pq}$. The contribution of the hidden parameters to the entry $(p,q)$ contains the factor $\frac{1}{n_L}\sum_{j}W^{(L+1)}_{pj}W^{(L+1)}_{qj}\,\cdots$, which tends to $0$ when $p\neq q$ by the law of large numbers, since the rows $p$ and $q$ of $W^{(L+1)}$ are independent and centered.
\end{proof}

\begin{corollary}[\textbf{Screening at infinite width}]
	\label{cor:ntk}
	Under Proposition~\ref{prop:ntk-diag}, one has $\rho=0$ and
	\begin{equation}
		\label{eq:Linf-ntk}
		L_\infty=\widehat{\vartheta}\,S_{DD}S_{DD}\he,
		\qquad
		\cond(\Gh)=\dfrac{b^{2}\xi^{2}}{\sigma_{\min}(S_{DD})^{2}}\bigl(1+o(1)\bigr).
	\end{equation}
\end{corollary}

\begin{proof}
	With $\Kh=\widehat{\vartheta}I_n$, one has $\Kh_{CD}=0$ and $\Schur{\Kh}{\Kh_{CC}}=\widehat{\vartheta}I_{n-m}$. Corollaries~\ref{cor:spectre} and \ref{cor:rho} then give \eqref{eq:Linf-ntk}.
\end{proof}

In the strict tangent kernel regime, the second screening is therefore inactive for a fully connected network, whatever the sharing of the hidden layers. A nonzero canonical correlation $\rho$ arises from the finite width, or from architectures whose output heads are correlated by construction. We denote by $\rho_{\mathrm{eff}}$ the quantity $\rho$ of Corollary~\ref{cor:rho} computed on the empirical tangent kernel of a network at initialization. Section~\ref{ssec:num-stat} measures $\rho_{\mathrm{eff}}=O(M^{-1/2})$ for a shared trunk of width $M$.
	
	\subsubsection{Condition number}
	\label{ssec:coro-cond}
	
	\begin{corollary}[\textbf{Condition number}]
	\label{cor:cond}
	Under the assumptions of Corollary~\ref{cor:spectre}, set
	\[
	\Lambda=\lmin\bigl(S_{DD}\Kh_{DD}S_{DD}\he\bigr),
	\qquad
	c_\star=\dfrac{\lmax(\Kh_{CC})\,\varepsilon_D^{2}\,\xi^{4}}{\Lambda}.
	\]
	Then, when $b\to\infty$ at fixed $\xi>0$,
	\begin{equation}
		\label{eq:cond-encadrement}
		c_\star\,\Pe^{2}\,\bigl(1+o(1)\bigr)
		\leq
		\cond(\Gh)
		\leq
		\dfrac{c_\star}{1-\rho^{2}}\,\Pe^{2}\,\bigl(1+o(1)\bigr).
	\end{equation}
\end{corollary}

\begin{proof}
	By \eqref{eq:peclet-def}, $b^{2}\xi^{2}=\varepsilon_D^{2}\xi^{4}\Pe^{2}$. By Corollary~\ref{cor:spectre}, $\cond(\Gh)=b^{2}\xi^{2}\lmax(\Kh_{CC})\,\lmin(L_\infty)^{-1}(1+o(1))$. The two-sided bound \eqref{eq:cor-rho} gives $(1-\rho^{2})\Lambda\le\lmin(L_\infty)\le\Lambda$, whence \eqref{eq:cond-encadrement}.
\end{proof}

	The condition number is expressed at fixed frequency. Its growth depends
	jointly on the convection strength and on the spatial frequency. For a given
	strength $b$, the large condition number regime appears when
	$b\xi$ becomes large compared with $\norm{S}_2$.
	
	\subsubsection{Temporal frequencies}
\label{ssec:coro-temps}

For the parabolic symbol $S=E\xi^{2}+H+i\omega I$, write $S_{DD}^{0}=(E\xi^{2}+H)_{DD}$. Since $\Schur{\Kh}{\Kh_{CC}}\succeq\lmin(\Schur{\Kh}{\Kh_{CC}})I$, congruence by $S_{DD}$ gives
\begin{equation}
	\label{eq:temps}
	\lmin\bigl(L_\infty(\omega)\bigr)\ge\lmin\bigl(\Schur{\Kh}{\Kh_{CC}}\bigr)\,\sigma_{\min}\bigl(S_{DD}^{0}+i\omega I\bigr)^{2},
	\qquad
	\sigma_{\min}\bigl(S_{DD}^{0}+i\omega I\bigr)\ge|\omega|-\norm{S_{DD}^{0}}_2.
\end{equation}
The symbol factor thus grows like $\omega^{2}$. The limit rate also depends on the kernel through $\lmin(\Schur{\Kh}{\Kh_{CC}})$, which is evaluated at the frequency $(\omega,\xi)$. If this quantity decays faster than $\omega^{-2}$, as for a kernel with transform $e^{-\omega^{2}}$, the limit rate tends to zero at high temporal frequencies. The bound \eqref{eq:temps} gives a growth of the rate only under a lower bound of the kernel uniform in $\omega$. When $S_{DD}^{0}$ is real symmetric and $\Schur{\Kh}{\Kh_{CC}}=c_K I$ with $c_K>0$, one has exactly $L_\infty(\omega)=c_K\bigl((S_{DD}^{0})^{2}+\omega^{2}I\bigr)$. The following sections are placed on the slice $\omega=0$, which corresponds to the steady problems used in the experiments.

\subsection{Numerical validation}
\label{ssec:num-q1}

\subsubsection{Spectral asymptotics and discrete framework}
	\label{ssec:num-symbole}

	For system \eqref{eq:sys3-explicite}, the bare diffusive block is
	$S_{DD}=\bigl(\begin{smallmatrix}1.5987 & -0.5\\ -0.5 & 0.6974\end{smallmatrix}\bigr)$,
	and for $\Kh=I$ the limit \eqref{eq:main} has spectrum
	$(0.225555,\,3.31663)$.

	Figure~\ref{fig:symbole} presents the extreme eigenvalues and the
	condition number of $\Gh$ as a function of $\Pe$ for system \eqref{eq:sys3}
	with $\Kh=I$. The behaviors predicted by
	Corollaries~\ref{cor:spectre} and \ref{cor:cond} are observed. The value
	$\lmin$ stabilizes at
	\[
	\lmin(S_{DD}S_{DD}\tr)=0.225555,
	\]
	while $\lmax$ approaches the asymptote $b^{2}\xi^{2}$. Beyond the
	knee, the condition number exhibits a growth in $\Pe^{2}$.
	
	The discrete framework satisfies none of the assumptions of the frequency
	reduction. Table~\ref{tab:cond} and Figure~\ref{fig:discret} show
	that the same scaling laws are observed there. The value of $\lmin(G)$ stays between $2.964\times10^{-3}$ and $2.965\times 10^{-3}$ from $\Pe=10^{3}$ to $\Pe=10^{5}$, while
	the condition number is multiplied by $100$ at each decade. The mass of the
	eigenvector associated with the slow mode on the diffusive block reaches unity as soon as
	$\Pe=10^{3}$. The saturation value differs from the one obtained for the symbol,
	because the discrete matrix aggregates the contributions of the different frequencies
	represented by the B-spline basis functions. On the other hand, the limit value does not depend on $b$, and the smallest eigenvalue converges to it, from $2.834\times10^{-3}$ at $\Pe=10^{2}$ to $2.965\times10^{-3}$.
	
	\begin{figure}[H]
		\centering
		\includegraphics[width=\linewidth]{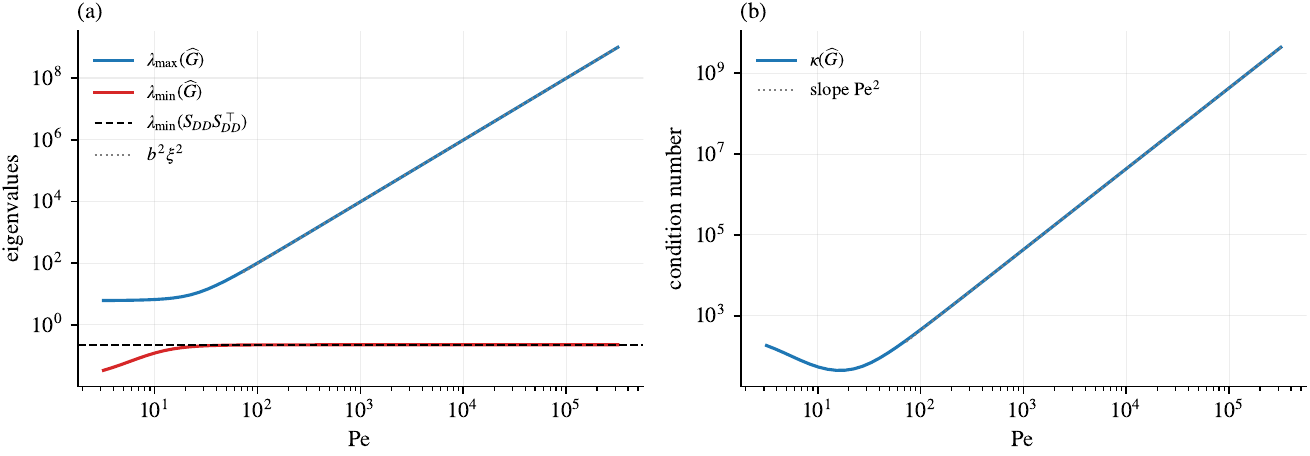}
		\caption{Exact symbol, system \eqref{eq:sys3}, $\Kh=I$. Left, eigenvalues
			of $\Gh$ with the saturation of $\lmin$ and the convective asymptote. Right,
			condition number with slope $\Pe^{2}$.}
		\label{fig:symbole}
	\end{figure}
	
	\begin{table}[H]
		\centering
		\caption{Discrete framework, $N=16$ B-splines per component. Spectrum of
			$G\in\R^{48\times 48}$ and mass of the slow eigenvector on the diffusive block.}
		\label{tab:cond}
		\begin{tabular}{rrrrr}
			\toprule
			$\Pe$ & $\lmin(G)$ & $\lmax(G)$ & $\cond(G)$ & $\norm{(v_{\min})_D}_2^{2}$\\
			\midrule
			$10^{2}$ & $2.834\times 10^{-3}$ & $1.14$ & $4.04\times 10^{2}$ & $0.9962$\\
			$10^{3}$ & $2.964\times 10^{-3}$ & $1.07\times 10^{2}$ & $3.60\times 10^{4}$ & $1.0000$\\
			$10^{4}$ & $2.965\times 10^{-3}$ & $1.07\times 10^{4}$ & $3.59\times 10^{6}$ & $1.0000$\\
			$10^{5}$ & $2.965\times 10^{-3}$ & $1.07\times 10^{6}$ & $3.59\times 10^{8}$ & $1.0000$\\
			\bottomrule
		\end{tabular}
	\end{table}
	
	\begin{figure}[H]
		\centering
		\includegraphics[width=\linewidth]{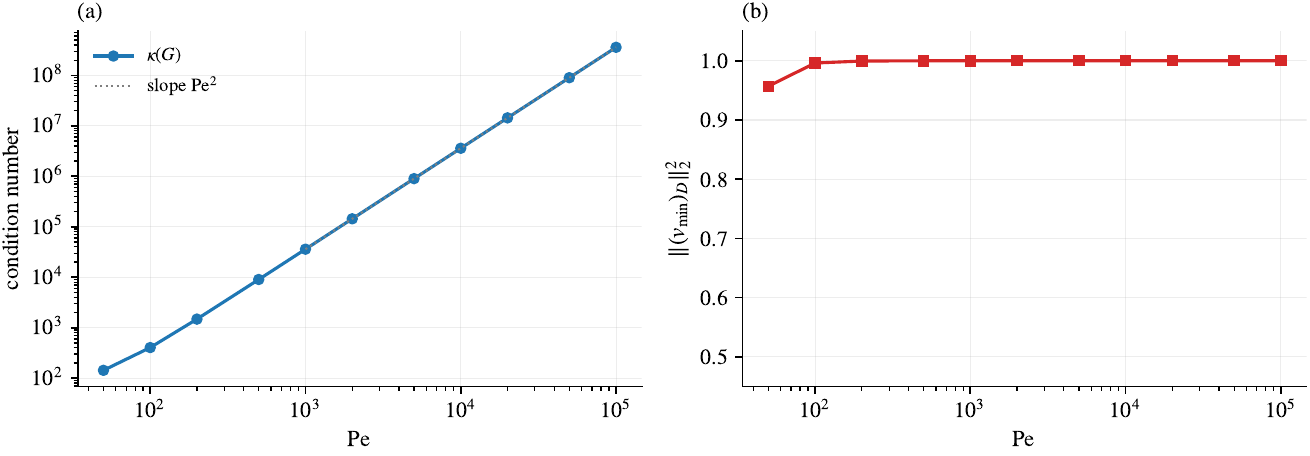}
		\caption{Discrete framework. Left, condition number of $G$. Right, mass of the
			slow eigenvector on the diffusive block.}
		\label{fig:discret}
	\end{figure}
	
	\subsubsection{Algebraic verifications}
	\label{ssec:num-alg}
	
	Three verifications concern the algebraic statements. For
	$(n,m)\in\{(3,1),(4,1),(5,2),(6,3)\}$, we draw
	$S=R+2.5\,I_n$, where $R$ is a Gaussian matrix, and
	$\Kh=BB\tr+0.3\,I_n$, where $B$ is a Gaussian matrix. We thus obtain
	dense couplings and a dense tangent kernel. Figure~\ref{fig:ncomp}
	presents the relative error between the eigenvalues of
	$\Schur{\Gh}{\Gh_{CC}}$ and those of its limit. The observed convergence is
	in $O(b^{-2})$, in agreement with Corollary~\ref{cor:ordre}.
	
	We then consider a family of kernels parameterized by its largest
	canonical correlation. For $b=10^{6}$, Figure~\ref{fig:rho} compares
	$\lmin(\Gh)$ with the limit value given by Theorem~\ref{thm:main} and with the
	bound $1-\rho^{2}$. The relative gap between the computed value and the limit
	value is at most $6.1\times10^{-14}$. The slowdown factor reaches
	$5.3$ for $\rho=0.9$ and $50.3$ for $\rho=0.99$.
	
	Table~\ref{tab:temps} finally verifies the parabolic case of Section~\ref{ssec:coro-temps}. For
	system \eqref{eq:sys3} with $\Kh=I$, the eigenvalues of the Schur
	complement at $b=10^{7}$ reproduce the additive shift with a relative error at most equal to $1.5\times10^{-8}$. For $\omega\neq0$ the symbol is complex, and this error decreases like $b^{-1}$, in agreement with Corollary~\ref{cor:ordre}. For a dense system of size $n=5$, with
	$m=2$ and $\omega=3$, for which $S$ is complex and non-normal, the general
	limit is attained with a relative error of $5.5\times10^{-8}$.
	
	\begin{figure}[H]
		\centering
		\includegraphics[width=0.58\linewidth]{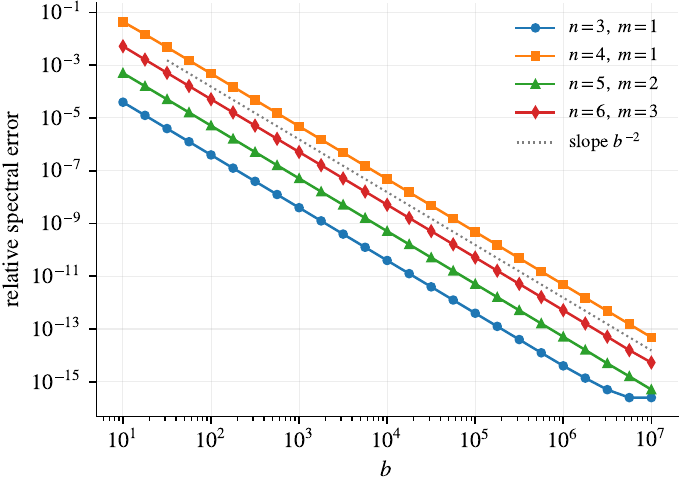}
		\caption{Dense random systems. Relative error on the eigenvalues
			between $\Schur{\Gh}{\Gh_{CC}}$ and its limit, with slope $O(b^{-2})$.}
		\label{fig:ncomp}
	\end{figure}
	
	\begin{figure}[H]
		\centering
		\includegraphics[width=0.55\linewidth]{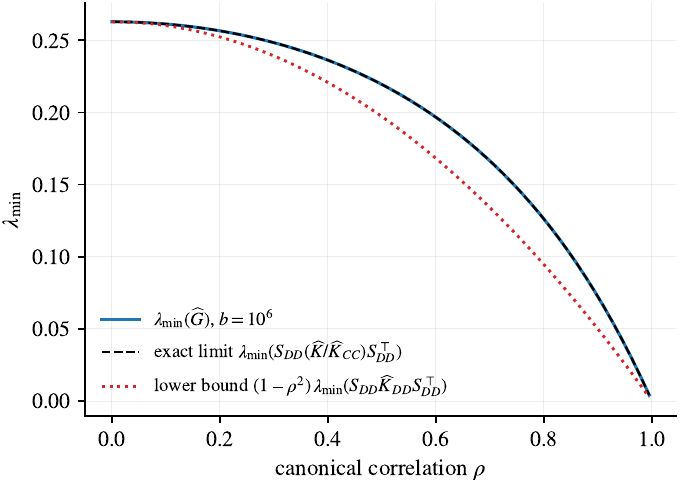}
		\caption{Screening of the kernel. Value $\lmin(\Gh)$ at $b=10^{6}$, limit of
			Theorem~\ref{thm:main} and bound $1-\rho^{2}$.}
		\label{fig:rho}
	\end{figure}
	
	\begin{table}[H]
		\centering
		\caption{Parabolic case, $b=10^{7}$. Eigenvalues of
			$\Schur{\Gh}{\Gh_{CC}}$ and prediction of Theorem~\ref{thm:main}.}
		\label{tab:temps}
		\begin{tabular}{llll}
			\toprule
			System & $\omega$ & computed eigenvalues & rel. error\\
			\midrule
			\eqref{eq:sys3} & $0$ & $(0.225555,\,3.31663)$ & $1.6\times 10^{-15}$\\
			\eqref{eq:sys3} & $2$ & $(4.225555,\,7.31663)$ & $1.5\times 10^{-8}$\\
			\eqref{eq:sys3} & $10$ & $(100.225555,\,103.31663)$ & $5.1\times 10^{-9}$\\
			dense $n=5$, $m=2$ & $3$ & $(9.70795,\,25.93154,\,176.16092)$ & $5.5\times 10^{-8}$\\
			\bottomrule
		\end{tabular}
	\end{table}
	
	\subsubsection{Statistical estimation of the exponents}
	\label{ssec:num-stat}
	
	The scaling laws manifest themselves as slopes on logarithmic
	diagrams, while part of the measured quantities is random. The
	systems of Section~\ref{ssec:num-alg} are drawn at random, and the tangent
	kernel of a finite width network depends on its initialization. We therefore
	re-evaluate these quantities according to a uniform protocol. Each experiment
	is repeated over forty independent draws. For each draw, a least squares
	regression is performed on the logarithmic diagram.
	The reported exponent is that of the median curve, together with a $95\%$
	confidence interval obtained by resampling the draws. The range of the
	exponents obtained over the different draws is indicated separately. The
	results are presented in Table~\ref{tab:exposants} and in
	Figure~\ref{fig:stat}.
	
	The two orders distinguished by Corollary~\ref{cor:ordre} appear
	distinctly. For the partition $(5,2)$, the gap in operator norm decreases
	with an exponent of $-1.0016$, with interval
	$[-1.0037,\,-1.0000]$, while the gap between the eigenvalues
	decreases with an exponent of $-2.0000$, to within less than $10^{-4}$. The three
	other partitions lead to the same results. The condition number of the
	symbol follows a law in $\Pe^{2}$, with an exponent of $1.99996$ and an
	interval $[1.99992,\,1.99999]$. The lower bound of
	\eqref{eq:cond-encadrement} is thus attained, the factor $1-\rho^{2}$
	modifying the multiplicative constant without modifying the exponent.
	
	On the slice $\Theta(\cdot,0.5)$ of the empirical tangent kernel, the relative gap
	in Frobenius norm with respect to the limit kernel decreases with an exponent of
	$-0.4738$, with interval $[-0.5089,\,-0.4566]$. The spurious correlation
	$\rho_{\mathrm{eff}}$ between the convective head and the diffusive heads decreases
	with an exponent of $-0.4832$, with interval
	$[-0.5450,\,-0.4375]$. Both intervals contain the value $-1/2$
	expected for a sampling fluctuation over $M$ neurons. The correlation considered in Corollary~\ref{cor:rho} thus appears as a finite width effect, while the limit kernel has decoupled blocks, in agreement with Proposition~\ref{prop:ntk-diag}.
	The penalty associated with the factor $1-\rho^{2}$ is then of order $O(1/M)$.
	
	\begin{table}[H]
		\centering
		\caption{Exponents measured on the median curve and $95\%$ intervals,
			forty draws. The last column gives the range across draws.}
		\label{tab:exposants}
		\small
		\footnotesize
		\begin{tabular}{llll}
			\toprule
			law & predicted & measured, $95\%$ interval & range across draws\\
			\midrule
			screening, $\norm{\cdot}_2$, $(5,2)$ & $-1$ & $-1.0016$ $[-1.0037,\,-1.0000]$ & $[-1.0039,\,-1.0000]$\\
			screening, eigenvalues, $(5,2)$ & $-2$ & $-2.0000$ $[-2.0000,\,-2.0000]$ & $[-2.0000,\,-1.9999]$\\
			condition number of the symbol & $2$ & $1.99996$ $[1.99992,\,1.99999]$ & $[1.99983,\,1.99999]$\\
			gap of the empirical kernel & $-1/2$ & $-0.4738$ $[-0.5089,\,-0.4566]$ & $[-0.673,\,-0.345]$\\
			correlation $\rho_{\mathrm{eff}}$ & $-1/2$ & $-0.4832$ $[-0.5450,\,-0.4375]$ & $[-0.698,\,-0.296]$\\
			\bottomrule
		\end{tabular}
	\end{table}
	
	\begin{figure}[H]
		\centering
		\includegraphics[width=\linewidth]{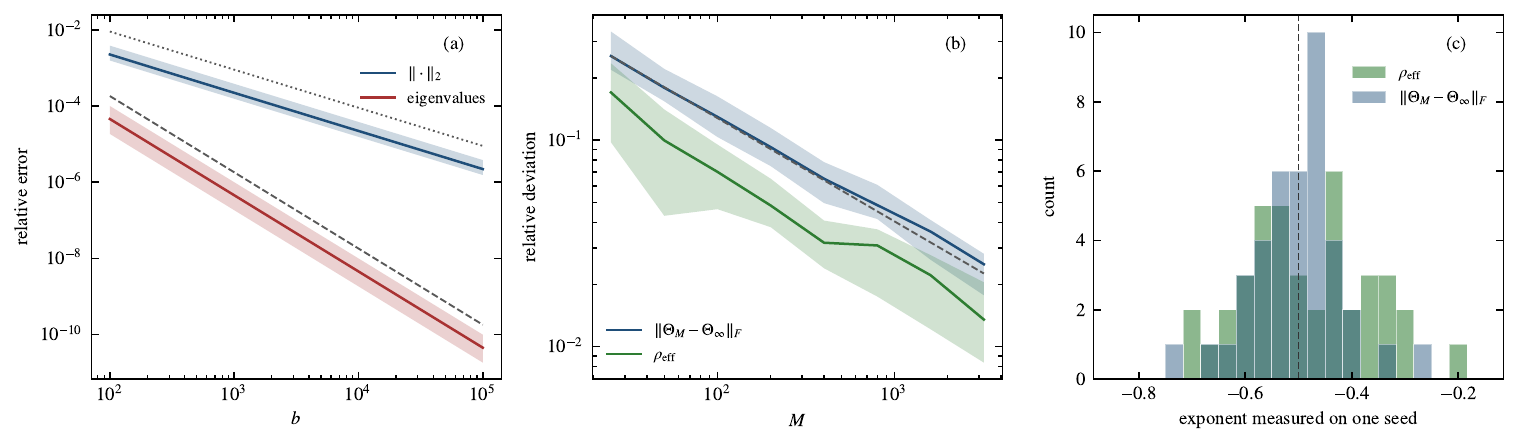}
		\caption{Statistical analysis over forty draws, median and interquartile
			band. (a) The two notions of error of Theorem~\ref{thm:main}.
			(b) Concentration of the empirical tangent kernel. (c) Distribution of the exponents
			obtained draw by draw.}
		\label{fig:stat}
	\end{figure}
	
	\section{The onset: Péclet threshold}
	\label{sec:seuil}
	
	This section answers (Q2). Corollary~\ref{cor:cond} describes the dominant convection regime. We now determine the Péclet number from which this regime is observed.
	
	\subsection{Frequency form}
	\label{ssec:seuil-freq}
	
	Corollary~\ref{cor:spectre} provides explicit expressions of both
	ends of the spectrum. We can thus define the Péclet threshold exactly and give a closed-form asymptotic approximation of it. Let us fix a factor $\tau>1$ and
	define the threshold as the Péclet number at which the condition number has been
	multiplied by $\tau$ with respect to its value without convection,
	$\cond_0 =\cond(\Gh)\big|_{b=0}$. The exact threshold is defined by
	\begin{equation}
		\label{eq:seuil-exact}
		b_\star(\tau)=\inf\bigl\{b>0\,:\,\cond\bigl(\Gh(b)\bigr)\ge\tau\,\cond_0\bigr\},
		\qquad
		\Pe_\star(\tau)=\dfrac{b_\star(\tau)}{\varepsilon_D\,\xi}\cdotp
	\end{equation}
	Since $\Gh(b)=\Gh_0+b\,\Gh_1+b^{2}\Gh_2$ with matrices independent of $b$, the exact threshold is obtained by a scalar root search, without any new assembly. In practice, the first crossing is located on a logarithmic grid and then refined by a root search on the bracketing interval. For the two-dimensional network of Section~\ref{ssec:num-2dtrain}, the refined values differ from the grid interpolation by less than $3\times10^{-6}$ in relative value on the eight initializations. The refinement applies to the first crossing located on the grid. A closed-form approximation is obtained by substituting the equivalents
	\[
	\lmax(\Gh)\simeq\lmax(\Kh_{CC})\,b^{2}\xi^{2},
	\qquad
	\lmin(\Gh)\simeq\lmin(L_\infty),
	\]
	in the preceding definition and solving for $b$. This gives
	\begin{equation}
		\label{eq:seuil-precis}
		\Pe^{\star}(\tau)
		=
		\dfrac{1}{\varepsilon_D\,\xi^{2}}
		\sqrt{
			\dfrac{\tau\,\cond_0\,\lmin(L_\infty)}
			{\lmax(\Kh_{CC})}
		} \cdotp
	\end{equation}
		The quantity $\Pe^{\star}(\tau)$ is an asymptotic approximation of $\Pe_\star(\tau)$, obtained by substituting the equivalents of Corollary~\ref{cor:spectre}. Sections~\ref{ssec:num-pred1d} to \ref{ssec:num-2dtrain} compare the two quantities. The factor $\tau$ fixes the level of degradation of the conditioning. This expression is entirely determined by the spectral quantities of the
	system and of the tangent kernel, without involving any fitting parameter.
	All the quantities appearing in it are independent of $b$. The block $S_{DD}$ is the
	bare diffusive symbol, the two quantities $\Schur{\Kh}{\Kh_{CC}}$ and
	$\lmax(\Kh_{CC})$ depend only on the tangent kernel, and $\cond_0$ is the
	condition number of the problem without convection. The tangent kernel enters the threshold explicitly, which makes it possible to distinguish two architectures applied to the same system, as long as their kernels differ by more than a scalar factor.
	
	\begin{corollary}[\textbf{Péclet threshold at infinite width}]
	\label{cor:ntk-seuil}
	Under Proposition~\ref{prop:ntk-diag}, the threshold \eqref{eq:seuil-precis} reduces to
	\begin{equation}
		\label{eq:seuil-ntk}
		\Pe^{\star}(\tau)=\dfrac{\sqrt{\tau}\;\cond(S)\;\sigma_{\min}(S_{DD})}{\varepsilon_D\,\xi^{2}} \cdotp
	\end{equation}
\end{corollary}

\begin{proof}
	With $\Kh=\widehat{\vartheta}I_n$, one has $\cond_0=\cond(SS\he)=\cond(S)^{2}$, $\lmax(\Kh_{CC})=\widehat{\vartheta}$ and, by Corollary~\ref{cor:ntk}, $\lmin(L_\infty)=\widehat{\vartheta}\,\sigma_{\min}(S_{DD})^{2}$. Inserting these values in \eqref{eq:seuil-precis} gives \eqref{eq:seuil-ntk}.
\end{proof}

At infinite width, the threshold depends only on the symbol of the system, the scalar kernel $\widehat{\vartheta}$ cancelling out. For system \eqref{eq:sys3}, formula \eqref{eq:seuil-ntk} gives $\Pe^{\star}(10)=292.4$, while the exact threshold \eqref{eq:seuil-exact} equals $291.7$. The dependence of \eqref{eq:seuil-precis} on the kernel thus expresses finite width effects, or architectures whose output heads are correlated by construction. The dense kernels used in Section~\ref{ssec:num-pred1d} model this situation.

In dimension $d\geq2$, replacing $b\xi$ by $b\,(\beta\cdot\mathbf{k})$ in \eqref{eq:seuil-precis} gives the threshold associated with
	each mode $\mathbf{k}\in\mathbb{R}^d$. The threshold then becomes anisotropic. At
	fixed $\lvert\mathbf{k}\rvert$, it grows when $\lvert\beta\cdot\mathbf{k}\rvert$ decreases, so
	that the modes almost orthogonal to the transport degrade at higher Péclet
	numbers.
	
	\subsection{Discrete form}
	\label{ssec:seuil-discret}
	
	Form \eqref{eq:seuil-precis} relies on Assumption~\ref{hyp:fourier}. A
	second form not appealing to this assumption applies directly to the
	Gram matrix $G=JJ\tr$ of a given network, where
	$J=\partial r/\partial\theta$ is the Jacobian of the residuals.
	
	Residual \eqref{eq:residu} is affine in the convection strength, since only the transport term depends on it. The Jacobian is therefore also affine, $J(b)=J_0+bJ_1$, and
\begin{equation}
	\label{eq:affine}
	x\tr G(b)\,x=\norm{J_0\tr x+b\,J_1\tr x}_2^{2},
	\qquad x\in\R^{N_r},
\end{equation}
where $N_r=nN$ is the number of residual rows and $p$ the number of parameters. The following statement is the discrete counterpart of Theorem~\ref{thm:main}. The residual space plays the role of the components and the parameter space plays the role of the kernel.

\begin{theorem}[\textbf{Discrete double screening}]
	\label{thm:discret}
	Let $J_0,J_1\in\R^{N_r\times p}$ with $J_1\neq0$ and $\mathcal{N}=\ker(J_1\tr)\neq\{0\}$. Let $\Pi$ be the orthogonal projector of $\R^{N_r}$ onto $\mathcal{N}$ and $Q_1$ the orthogonal projector of $\R^{p}$ onto $\operatorname{range}(J_1\tr)$. Set
	\begin{equation}
		\label{eq:mu-inf}
		\mu_\infty=\min_{x\in\mathcal{N},\ \norm{x}_2=1}\norm{(I_p-Q_1)\,J_0\tr x}_2^{2}.
	\end{equation}
	Then, when $b\to\infty$,
	\begin{equation}
		\label{eq:discret-lim}
		\lmax\bigl(G(b)\bigr)=b^{2}\lmax(J_1J_1\tr)\bigl(1+O(b^{-1})\bigr),
		\qquad
		\lmin\bigl(G(b)\bigr)=\mu_\infty+O(b^{-1}).
	\end{equation}
	Moreover
	\begin{equation}
		\label{eq:mu-compare}
		\mu_\infty\le\min_{x\in\mathcal{N},\ \norm{x}_2=1}\norm{J_0\tr x}_2^{2}=\lmin\bigl(\Pi J_0J_0\tr\Pi\big|_{\mathcal{N}}\bigr),
	\end{equation}
	and the inequality is strict as soon as $Q_1J_0\tr x\neq0$ for some minimizer $x$ of the right-hand side. Finally, $\mu_\infty>0$ if and only if
	\begin{equation}
		\label{eq:non-deg}
		\ker\bigl(J_0\tr|_{\mathcal{N}}\bigr)=\{0\}
		\quad\text{and}\quad
		J_0\tr(\mathcal{N})\cap\operatorname{range}(J_1\tr)=\{0\}.
	\end{equation}
\end{theorem}

\begin{proof}
	The bound on $\lmax$ follows from Weyl's theorem applied to $G(b)=b^{2}J_1J_1\tr+b(J_0J_1\tr+J_1J_0\tr)+J_0J_0\tr$.

	Let $\mathcal{N}^{\perp}=\operatorname{range}(J_1)$. The restriction $B=J_1\tr|_{\mathcal{N}^{\perp}}$ is a bijection of $\mathcal{N}^{\perp}$ onto $\operatorname{range}(J_1\tr)$, with smallest singular value $s_1>0$. For $x=x_N+x_\perp$ with $x_N\in\mathcal{N}$ and $x_\perp\in\mathcal{N}^{\perp}$, set $y=J_0\tr x+bBx_\perp$. One has
	\[
	(I_p-Q_1)\,y=(I_p-Q_1)J_0\tr x_N+(I_p-Q_1)J_0\tr x_\perp,
	\qquad
	Q_1y=Q_1J_0\tr x+bBx_\perp.
	\]
	Fix $x_N$ with $\norm{x_N}_2=1$. For $b>\norm{J_0}_2/s_1$, the operator $bB+Q_1J_0\tr|_{\mathcal{N}^{\perp}}$ is invertible from $\mathcal{N}^{\perp}$ onto $\operatorname{range}(J_1\tr)$, and the choice $x_\perp=-(bB+Q_1J_0\tr)^{-1}Q_1J_0\tr x_N$ gives $Q_1y=0$ and $\norm{x_\perp}_2=O(b^{-1})$. Hence
	\[
	\min_{x_\perp}x\tr G(b)x\le\norm{(I_p-Q_1)J_0\tr x_N}_2^{2}+O(b^{-1}).
	\]
	Conversely, at a minimizer $x_\perp$, one has $\norm{Q_1y}_2^{2}\le x\tr G(b)x\le c$ with $c$ independent of $b$. Since $\norm{Q_1y}_2\ge bs_1\norm{x_\perp}_2-\norm{J_0}_2(1+\norm{x_\perp}_2)$, this gives $\norm{x_\perp}_2=O(b^{-1})$, and then
	\[
	x\tr G(b)x\ge\norm{(I_p-Q_1)y}_2^{2}\ge\norm{(I_p-Q_1)J_0\tr x_N}_2^{2}-O(b^{-1}).
	\]
	The minimum over $x_\perp$ is the quadratic form of the Schur complement of $G(b)$ relative to the block $\mathcal{N}^{\perp}$, by \eqref{eq:schur-var}. It converges to the compression of $J_0(I_p-Q_1)J_0\tr$ to $\mathcal{N}$, whose smallest eigenvalue is $\mu_\infty$. The convergence of $\lmin(G(b))$ to $\mu_\infty$ follows as in the proof of Corollary~\ref{cor:spectre}. Finally, $\norm{(I_p-Q_1)J_0\tr x}_2^{2}=\norm{J_0\tr x}_2^{2}-\norm{Q_1J_0\tr x}_2^{2}$, which gives \eqref{eq:mu-compare} and the strict inequality. Since $\mathcal{N}$ has finite dimension, $\mu_\infty=0$ if and only if there exists $x\in\mathcal{N}\setminus\{0\}$ with $J_0\tr x\in\operatorname{range}(J_1\tr)$. This happens if and only if one of the two conditions \eqref{eq:non-deg} fails.
\end{proof}

Condition \eqref{eq:non-deg} is a rank condition, of the same nature as the positivity criteria of \citet{lyu2026}, but specific to the affine structure $J_0+bJ_1$. It is not automatic. For $J_0=\bigl(\begin{smallmatrix}0&1\\1&0\end{smallmatrix}\bigr)$ and $J_1=\bigl(\begin{smallmatrix}1&0\\0&0\end{smallmatrix}\bigr)$, one has $G(0)=I_2$ and $\mu_\infty=0$. In that case $\lmin(G(b))=b^{-2}(1+o(1))$ and $\cond(G(b))=b^{4}(1+o(1))$. When \eqref{eq:non-deg} fails, the condition number therefore grows faster than $b^{2}$, and the threshold formula below does not apply. The projector $I_p-Q_1$ is the discrete form of the second screening. The same orthogonal elimination appears in separable nonlinear least squares, where the linear variables are eliminated by projection \citep{golub1973}. What is specific here is its role in the limit $b\to\infty$ of the affine Gram matrix, and its link with the Péclet threshold. It removes from the gradients of the diffusive residuals their component along the gradients of the transport residuals. In the frequency setting, the latter correspond to the block $\Kh_{CC}$, and $\mu_\infty$ corresponds to $\lmin(S_{DD}(\Schur{\Kh}{\Kh_{CC}})S_{DD}\he)$. Restricting $J_0J_0\tr$ to $\mathcal{N}$ without this projector would amount to replacing $\Schur{\Kh}{\Kh_{CC}}$ by $\Kh_{DD}$, and would overestimate the limit rate and the threshold. The threshold is defined relative to $\cond_0=\cond(G(0))$, which is finite only when $G(0)=J_0J_0\tr$ is invertible. Condition \eqref{eq:non-deg} does not imply it. For $J_0=\diag(0,1)$ and $J_1=\diag(1,0)$, one has $\mu_\infty=1$ while $G(0)$ is singular. We therefore assume the following rank condition:
	\begin{equation}
		\label{eq:rang-J0}
		\rang(J_0)=N_r.
	\end{equation}
	Under conditions \eqref{eq:non-deg} and \eqref{eq:rang-J0}, the two quantities \eqref{eq:discret-lim} are carried into the definition of the threshold. This gives
\begin{equation}
	\label{eq:seuil-discret}
	\Pe^{\star}(\tau)=\dfrac{1}{\varepsilon_D\,\xi}
	\sqrt{\dfrac{\tau\,\cond_0\,\mu_\infty}{\lmax(J_1J_1\tr)}},
\end{equation}
where $\cond_0$ is this time the condition number of $G(0)$ and $\xi$ the reference frequency used to convert $b$ into a Péclet number. We write $b^{\star}=\varepsilon_D\,\xi\,\Pe^{\star}(\tau)$, so that $(b^{\star})^{2}=\tau\cond_0\mu_\infty/\lmax(J_1J_1\tr)$.

The approximation $b^{\star}$ of the exact threshold \eqref{eq:seuil-exact} admits an explicit non-asymptotic bracket. Set $G_0=J_0J_0\tr$, $G_1=J_0J_1\tr+J_1J_0\tr$, $\lambda_2=\lmax(J_1J_1\tr)$, and let $s_1$ be the smallest nonzero singular value of $J_1$. For $b\,s_1>\sqrt{\mu_\infty}+\norm{J_0}_2$, define
\begin{equation}
	\label{eq:l-u}
	\begin{gathered}
	t_b=\dfrac{\sqrt{\mu_\infty}+\norm{J_0}_2}{b\,s_1},
	\qquad
	\ell(b)=\Bigl(\sqrt{\mu_\infty}\sqrt{1-t_b^{2}}-\norm{J_0}_2\,t_b\Bigr)_{+}^{2},
	\\
	u(b)=\Bigl(\sqrt{\mu_\infty}+\dfrac{\norm{J_0}_2^{2}}{b\,s_1-\norm{J_0}_2}\Bigr)^{2}.
	\end{gathered}
\end{equation}

\begin{proposition}[\textbf{Bracket of the threshold}]
	\label{prop:encadrement}
	Under the assumptions of Theorem~\ref{thm:discret} and conditions \eqref{eq:non-deg} and \eqref{eq:rang-J0}, one has, for every $b$ with $b\,s_1>\sqrt{\mu_\infty}+\norm{J_0}_2$,
	\begin{equation}
		\label{eq:lmin-bracket}
		\ell(b)\le\lmin\bigl(G(b)\bigr)\le u(b),
		\qquad
		b^{2}\lambda_2-b\norm{G_1}_2\le\lmax\bigl(G(b)\bigr)\le b^{2}\lambda_2+b\norm{G_1}_2+\norm{G_0}_2.
	\end{equation}
	Let $b_0>0$ satisfy $b_0\,s_1>\sqrt{\mu_\infty}+\norm{J_0}_2$ and $b_0\lambda_2>\norm{G_1}_2$. Then every $b\ge b_0$ such that $\cond(G(b))=\tau\cond_0$ satisfies
	\begin{equation}
		\label{eq:bracket}
		\left(\dfrac{\ell(b_0)/\mu_\infty}{1+\frac{\norm{G_1}_2}{b_0\lambda_2}+\frac{\norm{G_0}_2}{b_0^{2}\lambda_2}}\right)^{1/2}
		\le
		\dfrac{b}{b^{\star}}
		\le
		\left(\dfrac{u(b_0)/\mu_\infty}{1-\frac{\norm{G_1}_2}{b_0\lambda_2}}\right)^{1/2}.
	\end{equation}
		Both bounds equal $1+O(b_0^{-1})$. If moreover $\cond(G(b))<\tau\cond_0$ for every $b\in[0,b_0]$, then the exact threshold satisfies $b_\star(\tau)\ge b_0$, and \eqref{eq:bracket} holds for $b=b_\star(\tau)$.
\end{proposition}

\begin{proof}
	The bounds on $\lmax$ follow from Weyl's theorem and from $G_0\succeq0$. We use the notation of the proof of Theorem~\ref{thm:discret}. Let $x=x_N+x_\perp$ be a unit vector and $t=\norm{x_\perp}_2$. By definition of $\mu_\infty$, $\norm{(I_p-Q_1)J_0\tr x_N}_2\ge\sqrt{\mu_\infty}\sqrt{1-t^{2}}$. Hence
	\[
	\norm{y}_2\ge\norm{(I_p-Q_1)y}_2\ge\sqrt{\mu_\infty}\sqrt{1-t^{2}}-\norm{J_0}_2\,t,
	\qquad
	\norm{y}_2\ge\norm{Q_1y}_2\ge b\,s_1\,t-\norm{J_0}_2.
	\]
	If $t\ge t_b$, the second inequality gives $\norm{y}_2\ge\sqrt{\mu_\infty}$. If $t<t_b$, the first one gives $\norm{y}_2\ge\sqrt{\mu_\infty}\sqrt{1-t_b^{2}}-\norm{J_0}_2t_b$. In both cases $x\tr G(b)x=\norm{y}_2^{2}\ge\ell(b)$. For the upper bound, let $x_N$ be a unit minimizer in \eqref{eq:mu-inf}, and let $x_\perp$ be the vector of the proof of Theorem~\ref{thm:discret}, which gives $Q_1y=0$. It satisfies $(b\,s_1-\norm{J_0}_2)\norm{x_\perp}_2\le\norm{J_0}_2$. Since $\norm{x}_2\ge1$,
	\[
	\lmin\bigl(G(b)\bigr)\le\norm{(I_p-Q_1)J_0\tr(x_N+x_\perp)}_2^{2}\le\bigl(\sqrt{\mu_\infty}+\norm{J_0}_2\norm{x_\perp}_2\bigr)^{2}\le u(b).
	\]
	Let now $b\ge b_0$ with $\cond(G(b))=\tau\cond_0=(b^{\star})^{2}\lambda_2/\mu_\infty$. By \eqref{eq:lmin-bracket},
	\[
	\dfrac{b^{2}\lambda_2-b\norm{G_1}_2}{u(b)}\le\dfrac{(b^{\star})^{2}\lambda_2}{\mu_\infty}\le\dfrac{b^{2}\lambda_2+b\norm{G_1}_2+\norm{G_0}_2}{\ell(b)}.
	\]
	The function $\ell$ is nondecreasing and $u$ is nonincreasing on the admissible range. The ratios $\norm{G_1}_2/(b\lambda_2)$ and $\norm{G_0}_2/(b^{2}\lambda_2)$ are decreasing. Replacing $b$ by $b_0$ in these four quantities preserves the two inequalities, which gives \eqref{eq:bracket}. Finally $t_{b_0}=O(b_0^{-1})$, so that $\ell(b_0)/\mu_\infty$ and $u(b_0)/\mu_\infty$ equal $1+O(b_0^{-1})$. Under the additional condition, the first crossing $b_\star(\tau)$ cannot lie in $[0,b_0]$, so that it lies in $[b_0,\infty)$ and the bracket applies to it.
\end{proof}

The proposition brackets the crossings located in the regime $b\ge b_0$. The closed formula \eqref{eq:seuil-discret} approximates the first crossing only under the additional condition. For $J_0=I_2$ and $J_1=\diag(-1,0)$, conditions \eqref{eq:non-deg} and \eqref{eq:rang-J0} hold with $\mu_\infty=1$, and $b^{\star}=\sqrt{\tau}$. The matrix $J(b)$ is singular at $b=1$, and the first crossing is $b_\star(\tau)=1-\tau^{-1/2}$, that is $0.684$ against $b^{\star}=3.162$ for $\tau=10$. Without a monotonicity property of $b\mapsto\cond(G(b))$, the first crossing may occur before $b_0$. The additional condition excludes this case. It is checked numerically on a fine grid of the curve $b\mapsto\cond(G_0+bG_1+b^{2}G_2)$, at no further assembly cost. The bracket is informative only when $b_0$ can be taken of the order of $b^{\star}$, that is when $(\sqrt{\mu_\infty}+\norm{J_0}_2)/s_1$ is small compared with $b^{\star}$. On random Jacobians with $12$ rows, $30$ parameters and four convective rows, the choice $b_0=b^{\star}/2$ is admissible and gives the bracket $[0.966,\,1.026]$ for $\tau=10^{5}$, and $[0.997,\,1.003]$ for $\tau=10^{7}$. The ratio between the exact threshold and $b^{\star}$ lies inside it in all cases. For the two-dimensional network of Section~\ref{ssec:num-2dtrain}, one has $s_1=4.6\times10^{-7}$ and $\norm{J_0}_2/s_1\approx2.9\times10^{4}$, while $b^{\star}\approx46$. The bracket is then inactive at the threshold, since $\mu_\infty\approx7.2\times10^{-13}$ while $\norm{J_0}_2^{2}\approx1.8\times10^{-4}$. An effective bracket for such Jacobians requires relative perturbation bounds adapted to their graded structure. The exact threshold \eqref{eq:seuil-exact} remains computable at the cost of a scalar root search. All the quantities are evaluated from two assemblies, at $b=0$ and $b=1$. The computation of the threshold therefore does not require the assembly of the Gram matrix $G(b)$ in the dominant convection regime.

We take $\tau=10$ in all the experiments and for all the systems compared.
	
	\subsection{Numerical validation}
\label{ssec:num-q2}

\subsubsection{Threshold in dimension one}
	\label{ssec:num-pred1d}
	
	We compare expression \eqref{eq:seuil-precis} with the exact threshold \eqref{eq:seuil-exact} on
	systems $A$ and $B$. The exact threshold is the Péclet number at which the
	condition number of $\Gh$ has increased tenfold with respect to its value at $b=0$, located on a fine grid and refined by interpolation.
	
	Expression \eqref{eq:seuil-precis} gives one value per kernel. We evaluate it
	for the two systems and eleven kernels, the identity and ten dense draws, and then
	we compute the exact threshold. Over the twenty-two cases, the ratio between
	the computed threshold and the exact threshold has median $1.0040$ and remains between
	$1.0007$ and $1.0214$. The thresholds range from $226$ to $2\,043$.
	Figure~\ref{fig:prediction} compares these values.
	
	We finally apply \eqref{eq:seuil-precis} to the empirical tangent kernel of a
	network with a shared trunk, evaluated at the center of the domain at initialization. It
	gives a threshold of $790$ for system $A$, with a range from $699$ to $881$
	depending on the initialization. We then train, on either side of this
	value, a shared trunk and three disjoint subnetworks of equal total width,
	over ten initializations, the learning rate being swept at each
	Péclet number and for each architecture. The paired ratio between the
	numbers of iterations equals $1.72$ at $\Pe=80$, $2.11$ at $250$, $4.00$ at
	$800$, $2.28$ at $2\,500$ and $1.34$ at $8\,000$. It is maximal in the neighborhood
	of the computed threshold and decreases beyond it.
	
	\begin{figure}[H]
		\centering
		\includegraphics[width=\linewidth]{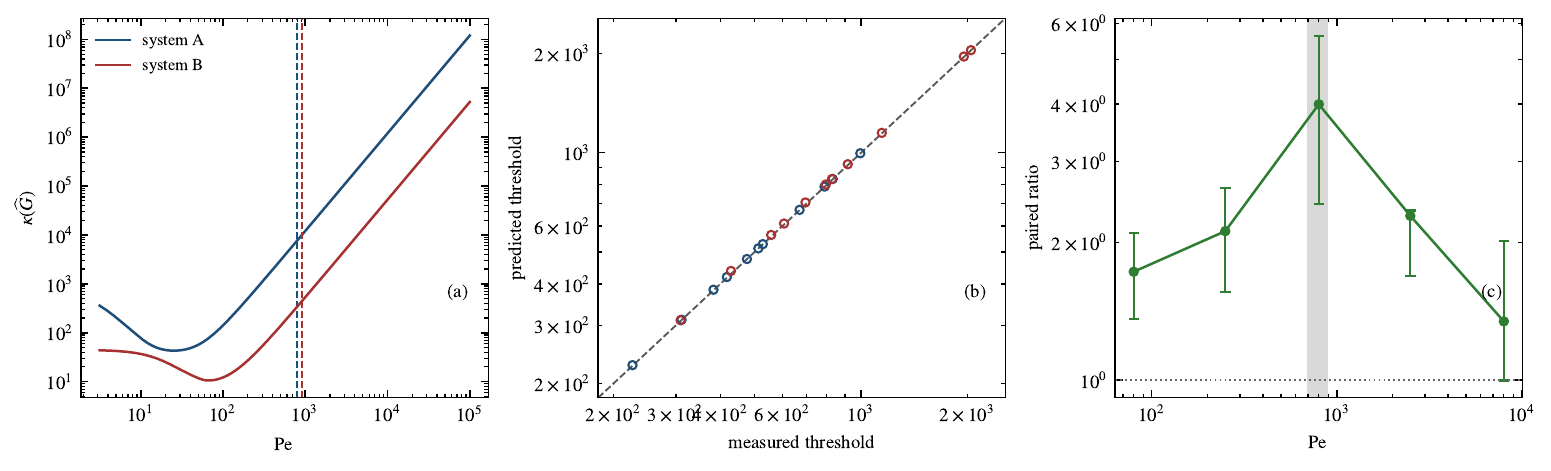}
		\caption{Star systems. (a) Condition number of the symbol, the vertical
			lines marking the thresholds given by \eqref{eq:seuil-precis}. (b) Computed
			threshold against exact threshold, two systems and eleven kernels. (c) Paired ratio
			between architectures, the shaded band indicating the computed threshold.}
		\label{fig:prediction}
	\end{figure}
	
	\subsubsection{Threshold in dimension two}
	\label{ssec:num-2d}
	
	In dimension two, the threshold depends on the mode and on its orientation relative to the
	transport direction, which provides a family of values to compare. We
	consider a four-component star exchanger on $(0,1)^{2}$, with
	diffusivities
	$(2\times10^{-3},1.5\times10^{-2},6\times10^{-3},3\times10^{-2})$ and exchange
	rates $(1.2,0.6,0.3)$, the first two components being
	convected. The approximation and the exact threshold are computed on the modes $\mathbf{k}=(j\pi,l\pi)$ with $1\leq j,l\leq6$, for two transport directions and six kernels, that is a total
	of four hundred and thirty-two symbol configurations.
	
	Figure~\ref{fig:twod} presents the results. The exact thresholds range from $145$ to $3\,148$, their variation being
	determined jointly by $\lvert\mathbf{k}\rvert$ and
	$\lvert\beta\cdot\mathbf{k}\rvert$. The ratio between the computed threshold and the exact
	threshold has median $1.0025$ and remains between $1.0011$ and
	$1.0076$. The two transport directions lead to a comparable
	accuracy.
	
	In the two-dimensional discrete framework, the condition number of the Gram matrix
	grows with an exponent of $2.0000$. The exact threshold equals $354$, against $505$
	for \eqref{eq:seuil-precis} evaluated on the fundamental mode, that is a ratio
	of $1.43$. The frequency form thus preserves the right order of magnitude
	outside the framework in which it is established, with however a lower accuracy
	than the one obtained on the symbol.
	
	\begin{figure}[H]
		\centering
		\includegraphics[width=\linewidth]{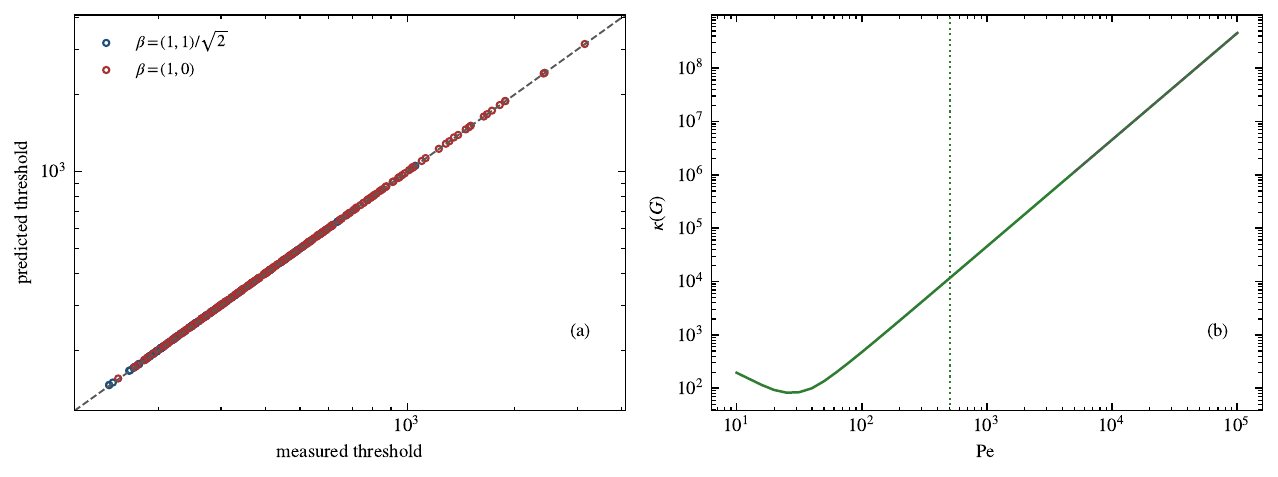}
		\caption{Two-dimensional star exchanger. (a) Computed threshold against exact threshold for thirty-six modes, two directions and six kernels.
			(b) Condition number in the discrete framework on the square, with the slope
			$\Pe^{2}$ and the threshold of the fundamental mode.}
		\label{fig:twod}
	\end{figure}
	
	\subsubsection{Threshold of the training operator}
	\label{ssec:num-2dtrain}
	
	The preceding verifications concern either the symbol, or a Gram matrix
	assembled from a prescribed kernel. We complete them with the study of
	the Gram matrix of an actually constructed network in dimension two, for
	which the discrete form \eqref{eq:seuil-discret} is directly applicable.
	
	The network is a perceptron with one hidden layer of width $M=80$, with hyperbolic
	tangent activation, with four output heads on a shared trunk. The
	boundary conditions on $(0,1)^{2}$ are imposed through the lifting
	$g(x,y)=x(1-x)y(1-y)$. The test functions consist of the tensor
	products of three cubic B-splines per direction, that is nine test functions
	per component and thirty-six residual rows in total. The system considered is
	the star exchanger of Section~\ref{ssec:num-2d}. The Jacobian is
	computed analytically, so that the Gram matrix obtained is that of the
	dynamics \eqref{eq:flot-r}, without numerical approximation of the Jacobian, for
	every convection strength.
	
	At each initialization, the kernel of $J_1\tr$ is of dimension eighteen, that is the number of residual rows associated with the two diffusive components. The limit of Theorem~\ref{thm:discret} is thus determined directly by the geometry of the Jacobian. The residual directions belonging to $\ker(J_1\tr)$ receive no contribution from the convective part, and the projector $I_p-Q_1$ removes from their gradients the part aligned with the transport gradients. At every initialization, $\mu_\infty>0$, so that condition \eqref{eq:non-deg} holds. The smallest eigenvalue $\lmin(G(b))$ is computed as $\sigma_{\min}(J(b))^{2}$ by a singular value decomposition of $J(b)=J_0+bJ_1$, without forming $G(b)$. By the backward stability of the singular value decomposition and Weyl's inequality, the absolute error on $\sigma_{\min}$ is of order $\epsilon_{\mathrm{mach}}\norm{J(b)}_2\approx1.2\times10^{-13}$, up to a modest constant. This estimate gives a relative error of order $3\times10^{-7}$ on $\lmin$. Over eight initializations, the relative gap between $\lmin(G(b))$ and $\mu_\infty$ is at most $8.1\times10^{-4}$ at $b=10^{3}$, $8.1\times10^{-5}$ at $b=10^{4}$ and $8.1\times10^{-6}$ at $b=10^{5}$, in agreement with the order $b^{-1}$ of Theorem~\ref{thm:discret}. The ratio between the threshold given by \eqref{eq:seuil-discret} and the exact threshold computed from $G(b)=G_0+bG_1+b^{2}G_2$ has median $1.0006$ and remains between $0.992$ and $1.004$. Replacing $\mu_\infty$ by the unscreened quantity $\lmin(\Pi J_0J_0\tr\Pi|_{\mathcal{N}})$ overestimates the limit rate by $4.8\%$ to $34\%$ depending on the initialization. The threshold is then overestimated systematically, with a ratio of median $1.075$ and range $[1.028,\,1.159]$. The second screening is therefore quantitatively visible on an actually constructed network, although its kernel is close to the diagonal limit of Proposition~\ref{prop:ntk-diag}.

In contrast, expression \eqref{eq:seuil-precis}, evaluated on the symbol at the
	fundamental mode, gives a ratio of $0.30$. This gap comes mainly
	from the difference between the condition numbers at zero convection used in
	the two computations: that of the Gram matrix of the network is of order
	$10^{10}$, while that of the $4\times4$ symbol is a few units. The
	quantity $\cond_0$ must therefore be computed for the same operator as the one whose
	departure is being measured. In the present case, the frequency form does not reproduce
	the condition number proper to the network and can therefore not provide a
	quantitative prediction of the observed threshold.
	
	Panel (c) of Figure~\ref{fig:twodtrain} reports the training cost
	by Adam, with a learning rate fixed by pretuning and a budget of
	$8\,000$ iterations, over eight initializations. The stopping criterion is
	\[
	\norm{r_D}/\norm{r_D^{0}}<10^{-2}.
	\]
	The median number of iterations equals $1\,609$ at $\Pe=300$, $3\,517$ at $\Pe=900$, $5\,524$ at $\Pe=2\,400$ and $7\,526$ at $\Pe=6\,000$, and then reaches the budget at $\Pe=15\,000$. The number of censored realizations goes from zero to six out of eight. In this experiment, the training cost therefore increases
	progressively with $\Pe$, without a clearly identifiable knee, and the exhaustion
	of the budget occurs beyond the range of the computed thresholds.
	
	\begin{figure}[H]
		\centering
		\includegraphics[width=\linewidth]{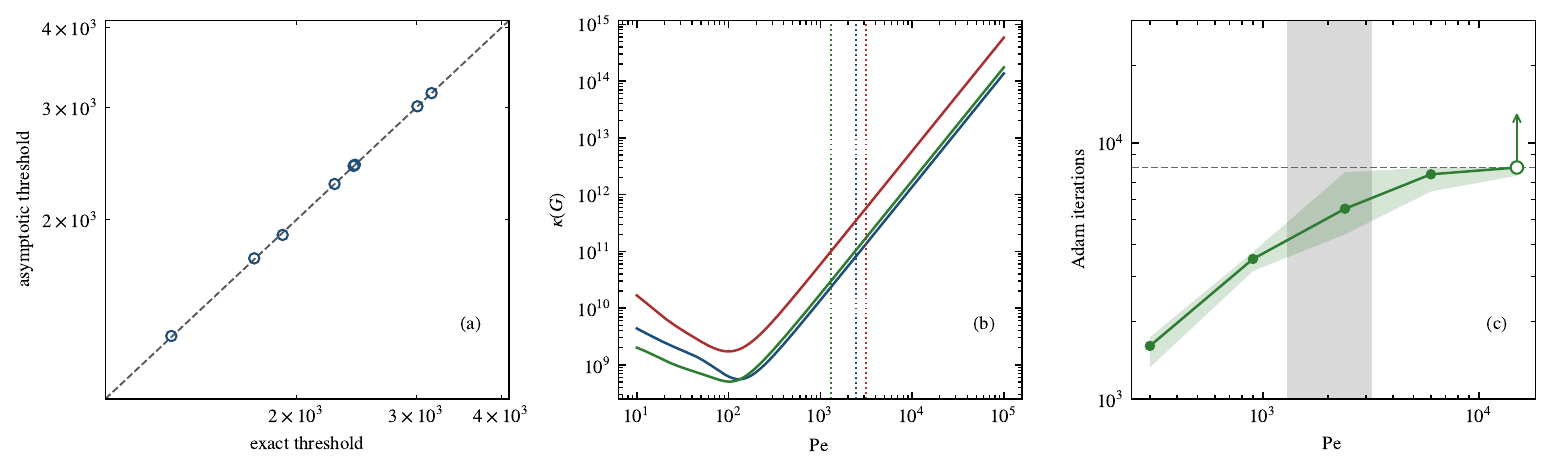}
		\caption{Two-dimensional variational network. (a) Threshold given by
			\eqref{eq:seuil-discret} against the exact threshold \eqref{eq:seuil-exact}, eight
			initializations. (b) Exact condition number of $G$ for three initializations.
			(c) Training cost by Adam, the shaded band indicating the range of the
			computed thresholds.}
		\label{fig:twodtrain}
	\end{figure}
	
	\section{The remedies: optimization}
	\label{sec:optim}
	
	\subsection{Continuous flow and discrete descent}

	\label{ssec:optim-flot}
	
	This section answers (Q3).
	
	The asymptotics \eqref{eq:lmin} does not contain $b$. The convergence rate of the
	slow mode is therefore asymptotically independent of the convection strength.
	Under continuous gradient flow, the diffusive component of the residual decreases like
	$\exp \left(-\lmin t \right)$. The difficulty appears at the discretization of the flow by a
	descent with step size $\eta$ shared among all the components.
	
	\begin{proposition}[\textbf{Training cost of the diffusive block}]
	\label{prop:iters}
	Consider the descent $r^{j+1}=(I-\eta G)r^{j}$ with $\eta=\lmax(G)^{-1}$, and assume that the initial residual has a nonzero component $c_{\min}$ on the slow mode. For $\vartheta\in\,]0,1[$, let $j(\vartheta)$ be the number of iterations needed to reduce the diffusive component of the residual by a factor $\vartheta$. Then
	\begin{equation}
		\label{eq:iters}
		j(\vartheta)\ge\bigl(\cond(G)-1\bigr)\,\log\dfrac{|c_{\min}|}{\vartheta\,\norm{r_D^{0}}_2+\norm{v_C}_2\,\norm{r^{0}}_2},
		\qquad
		j(\vartheta)\le\cond(G)\,\log\dfrac{\norm{r^{0}}_2}{\vartheta\,\norm{r_D^{0}}_2}+1.
	\end{equation}
	Here $v=(v_C,v_D)$ is the slow unit eigenvector, and the lower bound holds whenever the argument of the logarithm exceeds one. By Corollary~\ref{cor:spectre}, $\norm{v_C}_2=O(b^{-1})$. If $\norm{r^{0}}_2=o(b)$, the term $\norm{v_C}_2\norm{r^{0}}_2$ tends to zero. The initial residual may however grow like $b$, through the transport term $b\,\beta\cdot\nabla u_\theta$, and the term is then of order one and must be kept. Assume that
	\[
	\norm{v_C}_2\,\norm{r^{0}(b)}_2=o\bigl(\vartheta\,\norm{r_D^{0}(b)}_2\bigr)
	\quad\text{and}\quad
	\liminf_{b\to\infty}\dfrac{|c_{\min}(b)|}{\vartheta\,\norm{r_D^{0}(b)}_2}>1.
	\]
	Then the argument of the logarithm stays above a constant larger than one, and Corollary~\ref{cor:cond} gives, with the constant $c_\star$ of that corollary,
	\[
	c_\star\,\Pe^{2}\log\dfrac{|c_{\min}|}{\vartheta\norm{r_D^{0}}_2}\,\bigl(1+o(1)\bigr)\le j(\vartheta)\le\dfrac{c_\star}{1-\rho^{2}}\,\Pe^{2}\log\dfrac{\norm{r^{0}}_2}{\vartheta\norm{r_D^{0}}_2}\,\bigl(1+o(1)\bigr).
	\]
\end{proposition}

\begin{proof}
	In an orthonormal basis of eigenvectors of $G$, each component of the residual is multiplied at each iteration by $1-\eta\lambda\in[0,1-\cond(G)^{-1}]$. For the upper bound,
	\[
	\norm{r_D^{j}}_2\le\norm{r^{j}}_2\le\bigl(1-\cond(G)^{-1}\bigr)^{j}\norm{r^{0}}_2\le e^{-j/\cond(G)}\norm{r^{0}}_2,
	\]
	which is at most $\vartheta\norm{r_D^{0}}_2$ as soon as $j\ge\cond(G)\log(\norm{r^{0}}_2/(\vartheta\norm{r_D^{0}}_2))$. For the lower bound, let $v$ be the slow unit eigenvector, so that $v\he r^{j}=c_{\min}\bigl(1-\cond(G)^{-1}\bigr)^{j}$. Since $\norm{v_D}_2\le1$ and $\norm{r^{j}}_2\le\norm{r^{0}}_2$,
	\[
	\norm{r_D^{j}}_2\ge|v_D\he r_D^{j}|\ge|c_{\min}|\bigl(1-\cond(G)^{-1}\bigr)^{j}-\norm{v_C}_2\norm{r^{0}}_2.
	\]
	The condition $\norm{r_D^{j}}_2\le\vartheta\norm{r_D^{0}}_2$ therefore imposes
	\[
	\bigl(1-\cond(G)^{-1}\bigr)^{j}\le\dfrac{\vartheta\norm{r_D^{0}}_2+\norm{v_C}_2\norm{r^{0}}_2}{|c_{\min}|}.
	\]
	The lower bound follows from $-\log(1-x)\le x/(1-x)$ with $x=\cond(G)^{-1}$. The choice $\eta=\lmax(G)^{-1}$ is order optimal under the stability constraint $\eta<2/\lmax(G)$.
\end{proof}
	
	The difficulty is therefore a coupling between equations passing through the optimizer. The
	convective block imposes the step size constraint, which the diffusive block undergoes. This
	mechanism differs from the scalar spectral bias, which puts frequencies in
	competition within a single equation. Here the competition is between equations,
	at fixed frequency.
	
	\subsection{Selective weighting}
	\label{ssec:optim-pond}
	
	The theorem locates the cause in the largest eigenvalue carried by the
	convective block. Let us introduce the weighted loss
	\begin{equation}
		\label{eq:perte-w}
		\mathcal{J}_W(\theta)=\dfrac{1}{2}\sum_{k,i}w_i^{2}\,\bigl|r_{k,i}(\theta)\bigr|^{2},
		\qquad
		w_i=(b\xi_0)^{-1}\text{ if }i\in C,\qquad w_i=1\text{ if }i\in D,
	\end{equation}
	with $W=\diag(w_1,\dots,w_n)$ and Gram matrix $\Gh_W=W\Ah\Kh\Ah\he W$.
	
	\begin{proposition}[\textbf{Condition number under weighting}]
		\label{prop:weight}
		With $S_{DD}$ invertible and $\xi=\xi_0$,
		\begin{equation}
			\label{eq:GW-lim}
			\Gh_W\xrightarrow[b\to\infty]{}\widetilde{A}\,\Kh\,\widetilde{A}\he,
			\qquad
			\widetilde{A} =
			\begin{pmatrix}iI_m & 0\\ S_{DC} & S_{DD}\end{pmatrix},
		\end{equation}
		and $\widetilde{A}$ is invertible. Consequently, the condition number
		$\cond(\Gh_W)$ admits, as $b\to\infty$, a finite limit that does not depend
		on $\Pe$.
	\end{proposition}
	
	\begin{proof}
		The rows corresponding to the components of $C$ in $W\Ah$ read
		\[
		(iI_m+O(b^{-1})\ \ O(b^{-1})),
		\]
		while the rows corresponding to the components of $D$ remain equal to
		$(S_{DC}\ \ S_{DD})$. Hence,
		\[
		W\Ah\longrightarrow\widetilde{A}.
		\]
		The matrix $\widetilde{A}$ is block triangular, with diagonal blocks
		$iI_m$ and $S_{DD}$, which are invertible by assumption. It is therefore
		invertible. Since $\Kh\succ0$, we then have
		$\widetilde{A}\Kh\widetilde{A}\he\succ0$. The convergence of $W\Ah$ entails
		that of $\Gh_W$ to this matrix, and the continuity of the extreme
		eigenvalues gives the convergence of $\cond(\Gh_W)$ to a finite limit.
	\end{proof}
	
	\begin{remark}
		\label{rem:xi-choice}
		The weight depends on a reference frequency $\xi_0$, which is not known
		a priori in a general problem. Proposition~\ref{prop:weight} describes
		exactly the behavior of the weighting at this frequency. In dimension one, for a frequency $\xi\neq0$, the limit of $W\Ah$ is the matrix of Proposition~\ref{prop:weight} with $iI_m$ replaced by $i(\xi/\xi_0)I_m$. The condition number stays bounded on any band of frequencies bounded away from zero. In dimension $d\ge2$, the weighting is applied to all modes, including the modes with $\beta\cdot\mathbf{k}=0$, which receive no convective contribution. For such a mode with $S=I_2$, $\Kh=I_2$ and $W=\diag((b\xi_0)^{-1},1)$, the weighted Gram matrix equals $\diag((b\xi_0)^{-2},1)$, whose condition number is $b^{2}\xi_0^{2}$. A weight constant per component cannot be adapted to all orientations of the modes. The weighting is therefore effective on the modes with $|\beta\cdot\mathbf{k}|$ comparable to $\xi_0$, and degrades the modes almost orthogonal to the transport. We take for $\xi_0$ the frequency of the slowest mode, which dominates the diffusive residual. On the domain $(0,1)$, this frequency
		is the fundamental frequency of the Dirichlet Laplacian, namely
		$\xi_0=\pi$.
	\end{remark}
	
	\subsection{Architecture with disjoint subnetworks}
	\label{ssec:optim-archi}
	
	The second screening leads to an architectural question.
	
	By Corollary~\ref{cor:rho}, $\Schur{\Kh}{\Kh_{CC}}\preceq\Kh_{DD}$, with equality if and only if $\Kh_{CD}=0$. At fixed $S$, $\Kh_{CC}$ and $\Kh_{DD}$, the limit rate is therefore maximal when the kernel has decoupled blocks. The condition $\Kh_{DC}=0$ holds when the components of $C$ and of $D$ are associated with subnetworks using distinct parameters, since the gradients with respect to independent parameters are then orthogonal. By Proposition~\ref{prop:ntk-diag}, it also holds at infinite width for a shared trunk. For a trunk of finite width $M$, the cross blocks are of order $M^{-1/2}$, and the reduction factor $1-\rho^{2}$ differs from one by $O(1/M)$. The second screening alone therefore explains only a small part of the gaps between architectures measured in Section~\ref{ssec:num-adam}. The decisive difference lies in the interaction with the optimizer, studied in Section~\ref{ssec:optim-adam}. Architectures based on separate subnetworks are used in multiphysics physics-informed neural networks and in domain decomposition methods \citep{jagtap2020,moseley2023,heinlein2024}. Figure~\ref{fig:architectures} compares the two architectures.

\begin{figure}[H]
		\centering
		\includegraphics[width=0.92\linewidth]{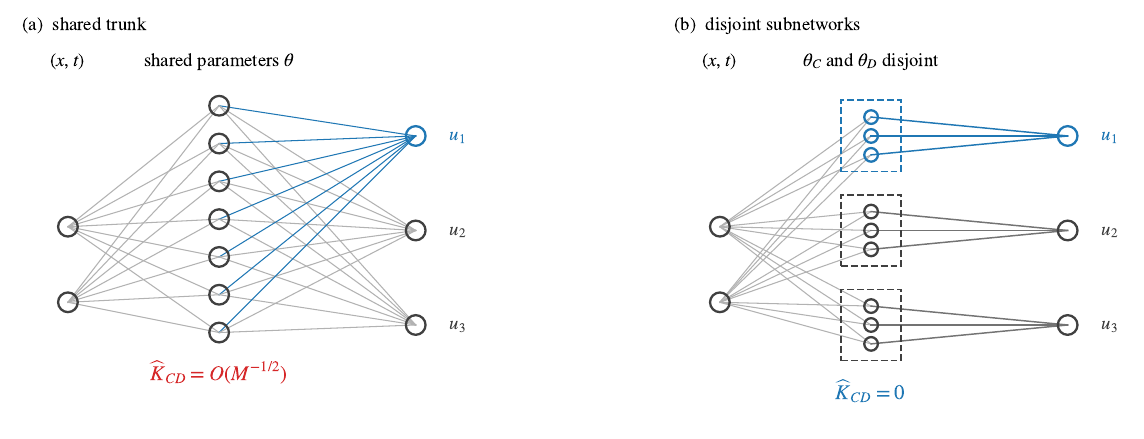}
		\caption{The two architectures compared. (a) Shared trunk, whose tangent kernel carries nonzero cross blocks at finite width, of order $M^{-1/2}$. (b) Disjoint subnetworks, whose
			tangent kernel has decoupled blocks.}
		\label{fig:architectures}
	\end{figure}
	
	\subsection{Interaction with the Adam optimizer}
	\label{ssec:optim-adam}
	
	Proposition~\ref{prop:iters} relates the training difficulty to the shared
	character of the optimization step size. Now physics-informed networks are
	commonly trained with the Adam optimizer \citep{kingma2015}, which applies an adaptive,
	coordinate-wise rescaling to the parameter vector. This rescaling
	acts in parameter space and not in residual space.
	It therefore does not correspond directly to a diagonalization of the Gram
	matrix.
	
	We model a preconditioned flow $\dot\theta=-\mathcal{D}\,\nabla_\theta \mathcal{J}$, where $\mathcal{D}$ is constant on groups of parameters. The parameters are partitioned into groups $g=1,\dots,q$, and $\Kh^{g}\succeq0$ denotes the contribution of group $g$ to the tangent kernel, so that $\Kh=\sum_g\Kh^{g}$. A factor $d_g>0$ on group $g$ replaces $\Kh$ by $\Kh_d=\sum_g d_g\Kh^{g}$, and the Gram matrix becomes $\Gh_d=\Ah\Kh_d\Ah\he$. A group is said to be convective if $\Kh^{g}_{DD}=0$, diffusive if $\Kh^{g}_{CC}=0$, and mixed otherwise. Positivity of $\Kh^{g}$ implies $\Kh^{g}_{CD}=0$ for convective and diffusive groups.

\begin{proposition}[\textbf{Group-wise preconditioning}]
	\label{prop:groupes}
	Assume $S_{DD}$ invertible and fix $\xi=\xi_0\neq0$.
	\begin{itemize}
		\item[(\textbf{i})] If every group is convective or diffusive, with $\Kh^{\mathrm{c}}_{CC}=\sum_{g\text{ conv.}}\Kh^{g}_{CC}\succ0$ and $\Kh^{\mathrm{d}}_{DD}=\sum_{g\text{ diff.}}\Kh^{g}_{DD}\succ0$, then the choice $d_g=(b|\xi_0|)^{-2}$ on convective groups and $d_g=1$ on diffusive groups gives
		\begin{equation}
			\label{eq:groupes-lim}
			\Gh_d\xrightarrow[b\to\infty]{}\widetilde{A}\,\widetilde{K}\,\widetilde{A}\he,
			\qquad
			\widetilde{A}=\begin{pmatrix}i\,\mathrm{sgn}(\xi_0)I_m & S_{CD}\\ 0 & S_{DD}\end{pmatrix},
			\qquad
			\widetilde{K}=\begin{pmatrix}\Kh^{\mathrm{c}}_{CC}&0\\0&\Kh^{\mathrm{d}}_{DD}\end{pmatrix},
		\end{equation}
		and $\cond(\Gh_d)$ has a finite limit independent of $\Pe$.
		\item[(\textbf{ii})] If a mixed group $g$ satisfies $\Kh^{g}_{CC}\neq0$, set $d_{\max}=\max_hd_h$. Then
		\[
		\cond(\Gh_d)\ge c\,\dfrac{d_g}{d_{\max}}\,b^{2}\xi_0^{2}\bigl(1+O(b^{-1})\bigr),
		\]
		with $c>0$ independent of $b$ and of the factors. A condition number bounded in $b$ therefore requires $d_g/d_{\max}=O(b^{-2})$. The contribution of the mixed group to the diffusive block of $\Gh_d$ is then negligible, of relative order $b^{-2}$, with respect to the contribution of the group with the largest factor.
	\end{itemize}
\end{proposition}

\begin{proof}
	(i) Write $\Ah=\begin{pmatrix}U_b&S_{CD}\\S_{DC}&S_{DD}\end{pmatrix}$ with $U_b=S_{CC}+ib\xi_0I_m$, and $\varrho=(b|\xi_0|)^{-1}$. Since the cross blocks of $\Kh_d$ vanish,
	\[
	\Gh_d=\begin{pmatrix}\varrho^{2}U_b\Kh^{\mathrm{c}}_{CC}U_b\he+S_{CD}\Kh^{\mathrm{d}}_{DD}S_{CD}\he & \varrho^{2}U_b\Kh^{\mathrm{c}}_{CC}S_{DC}\he+S_{CD}\Kh^{\mathrm{d}}_{DD}S_{DD}\he\\ \ast & \varrho^{2}S_{DC}\Kh^{\mathrm{c}}_{CC}S_{DC}\he+S_{DD}\Kh^{\mathrm{d}}_{DD}S_{DD}\he\end{pmatrix}.
	\]
	Since $\varrho U_b\to i\,\mathrm{sgn}(\xi_0)I_m$ and $\varrho^{2}U_b=O(b^{-1})$, each block converges to the corresponding block of $\widetilde{A}\widetilde{K}\widetilde{A}\he$. The matrix $\widetilde{A}$ is block triangular with invertible diagonal blocks and $\widetilde{K}\succ0$, so the limit is positive definite. The continuity of the extreme eigenvalues gives the finite limit of $\cond(\Gh_d)$.

	(ii) Let $x=(x_C,0)$ with $x_C$ a unit eigenvector of $\Kh^{g}_{CC}$ for $\lmax(\Kh^{g}_{CC})>0$. Then $(\Ah\he x)_C=(S_{CC}\he-ib\xi_0)x_C$ and $(\Ah\he x)_D=S_{CD}\he x_C$, so that
	\[
		\begin{aligned}
		\lmax(\Gh_d)\ge x\he\Gh_dx&\ge d_g\,(\Ah\he x)\he\Kh^{g}(\Ah\he x)\\
		&=d_g\,b^{2}\xi_0^{2}\,\lmax(\Kh^{g}_{CC})+O(d_gb).
	\end{aligned}
	\]
	For $x=(0,x_D)$ with $\norm{x_D}_2=1$, the vector $\Ah\he x=(S_{DC}\he x_D,S_{DD}\he x_D)$ does not depend on $b$. Hence $\lmin(\Gh_d)\le d_{\max}\norm{S}_2^{2}\norm{\Kh}_2$. The ratio gives the bound with the constant
	\[
	c=\dfrac{\lmax(\Kh^{g}_{CC})}{\norm{S}_2^{2}\,\norm{\Kh}_2}.
	\] If $d_g/d_{\max}=O(b^{-2})$, then $d_g\,S_{D\bullet}\Kh^{g}S_{D\bullet}\he=O(b^{-2})\,d_{\max}$.
\end{proof}

With disjoint subnetworks, every group is convective or diffusive, and case (i) applies. With a shared trunk, the trunk is a mixed group. By case (ii), a bounded condition number requires the trunk to be slowed down by a factor $b^{-2}$ relative to the fastest group. If, moreover, the parameters outside the trunk contribute a nondegenerate diffusive block, the diffusive components are then trained mainly by these parameters. The Adam optimizer rescales each parameter by the inverse of a running root mean square of its gradient. The gradient of the convective parameters carries the factor $b$ of the transport term, so that this rescaling acts in the direction of case (i). Adam's rescaling varies along training and between parameters of the same group, so that this argument is a hypothesis and not a convergence theorem. It suggests why the separation of the parameters can reduce the training cost by a factor much larger than $(1-\rho^{2})^{-1}$. The measurements of Section~\ref{ssec:num-adam} are consistent with it. A block preconditioner of Gauss--Newton type, analyzed for coupled systems in \citet{park2026}, realizes a block rescaling of the same nature.

	\subsection{Numerical validation}
\label{ssec:num-q3}

In all the training experiments, the right-hand sides are built from the manufactured solutions and from their derivatives computed analytically. On $(0,1)$, the manufactured solution of system \eqref{eq:sys3} is given by
\[
u_1(x)=\sin\pi x,\qquad u_2(x)=\sin\pi x+0.15\sin2\pi x,\qquad u_3(x)=e^{-x}\sin\pi x.
\]
It does not depend on $b$. The right-hand side is $f=-\diag(\varepsilon)u''+b\,P u'+Hu$, with $P=\diag(1,0,0)$. On $(0,1)^{2}$, the four components are $\sin\pi x\sin\pi y$, $u_2(x)\sin\pi y$, $u_3(x)\sin\pi y$ and $\sin\pi x\,\sin\pi y\,\bigl(1-0.2\sin2\pi y\bigr)$, and the right-hand side is obtained in the same way. The Gram matrices and the thresholds of Sections~\ref{ssec:num-q1} and \ref{ssec:num-q2} do not depend on the right-hand side.

\subsubsection{Gradient descent and weighting}
	\label{ssec:num-iters}
	
	In the discrete framework, we integrate the descent
	$r^{j+1}=(I-\eta G)r^{j}$ with $\eta=\lmax(G)^{-1}$, starting from an initial
	residual carried by the diffusive block, and we measure the number of iterations
	needed to reach
	\[
	\norm{r_D^{j}}/\norm{r_D^{0}}<10^{-3}.
	\]
	The weight uses $\xi_0=\pi$. Table~\ref{tab:iters} and Figure~\ref{fig:iters} compare the raw dynamics with the weighted dynamics. Without weighting, the cost follows the law
	$\Pe^{2}$ predicted by Proposition~\ref{prop:iters}, the measured ratio between
	$\Pe=10^{3}$ and $\Pe=10^{4}$ being $100.4$. With weighting, the number
	of iterations stabilizes at $701$, in agreement with
	Proposition~\ref{prop:weight}. This value corresponds to a convergence
	plateau and not to a measurement floor: the tolerance is reached in
	$701$ iterations for every $\Pe\geq5\times10^{2}$ considered, while the
	residual continues to decrease beyond that iteration.
	
	\begin{table}[H]
		\centering
		\caption{Descent iterations to reduce the diffusive residual by a factor
			$10^{3}$, computed for the linear descent on the fixed Gram matrix of the discrete framework, without network training.}
		\label{tab:iters}
		\begin{tabular}{rrrr}
			\toprule
			$\Pe$ & raw & weighted & speedup\\
			\midrule
			$10^{2}$ & $2\,755$ & $753$ & $3.7$\\
			$5\times 10^{2}$ & $60\,945$ & $701$ & $87$\\
			$10^{3}$ & $257\,291$ & $701$ & $367$\\
			$5\times 10^{3}$ & $6\,116\,267$ & $701$ & $8.7\times 10^{3}$\\
			$10^{4}$ & $25\,820\,765$ & $701$ & $3.7\times 10^{4}$\\
			\bottomrule
		\end{tabular}
	\end{table}
	
	\begin{figure}[H]
		\centering
		\includegraphics[width=0.58\linewidth]{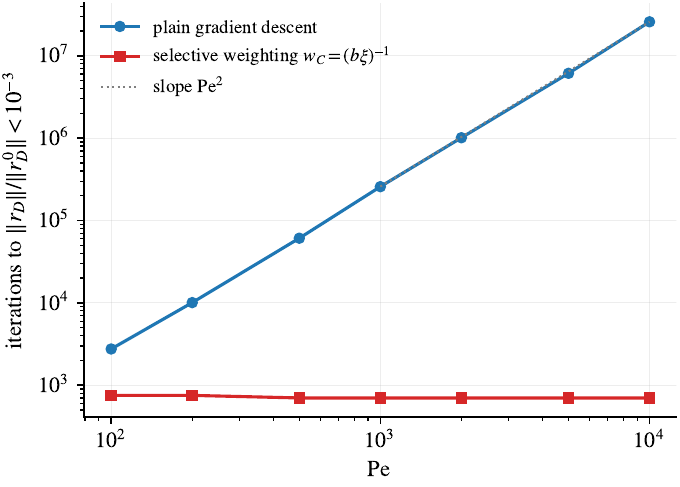}
		\caption{Training cost of the diffusive components, raw dynamics and
			selective weighting.}
		\label{fig:iters}
	\end{figure}
	
	\subsubsection{Finite width networks and full training}
	\label{ssec:num-largeur}
	
	We verify that the predictions remain consistent with the behavior of
	finite width networks that are actually trained. The networks considered are
	perceptrons with one hidden layer, with hyperbolic tangent activation, with
	three output heads on a shared trunk and boundary conditions
	imposed by lifting. Figure~\ref{fig:largeur} shows, on the left, the
	concentration of the empirical tangent kernel: the curves obtained over ten
	initializations converge to the reference curve when the width
	increases. On the right, the smallest eigenvalue of the Gram matrix $G_M$ of the network of width $M$ converges to its infinite width value.
	The effective correlation $\rho_{\mathrm{eff}}(M)$ decreases from $0.186$ at $M=25$ to
	$0.0149$ at $M=3200$, in agreement with the exponents of
	Table~\ref{tab:exposants}.
	
	Figure~\ref{fig:entrainement} presents the full training by gradient
	descent of a network of width $M=200$. The step size is $\eta=0.9/\lmax$, where $\lmax$ is estimated once by power iteration at initialization, and it is then kept constant. The loss is the variational
	loss of system \eqref{eq:sys3} with manufactured right-hand side, with
	$\Pe\in\{150,500,1500\}$ and six initializations. Without weighting, the curves
	separate according to $\Pe$ and reach a plateau. After $6\times10^{4}$
	iterations, the median final residual equals respectively $0.0106$, $0.0215$
	and $0.0899$, with interquartile intervals
	$[0.0097,0.0206]$, $[0.0178,0.0408]$ and $[0.0736,0.1636]$.
	The bands corresponding to the three regimes are disjoint at the end of the
	descent. With weighting, the three curves superpose and the median final
	residual equals respectively $0.0081$, $0.0081$ and $0.0080$.
	
	The paired ratio between the numbers of iterations of the two dynamics equals
	$1.29$, with an interval $[1.16,1.67]$, at $\Pe=150$, then $2.67$,
	with an interval $[2.17,3.15]$, at $\Pe=500$, and $9.87$, with an
	interval $[8.56,16.07]$, at $\Pe=1500$. Its growth with $\Pe$ is
	consistent with the law predicted by Proposition~\ref{prop:iters}. This agreement is observed in the nonlinear training, where the tangent kernel evolves. It shows that the trends
	predicted in the linearized regime remain observable in a full
	training.
	
	\begin{figure}[H]
		\centering
		\includegraphics[width=\linewidth]{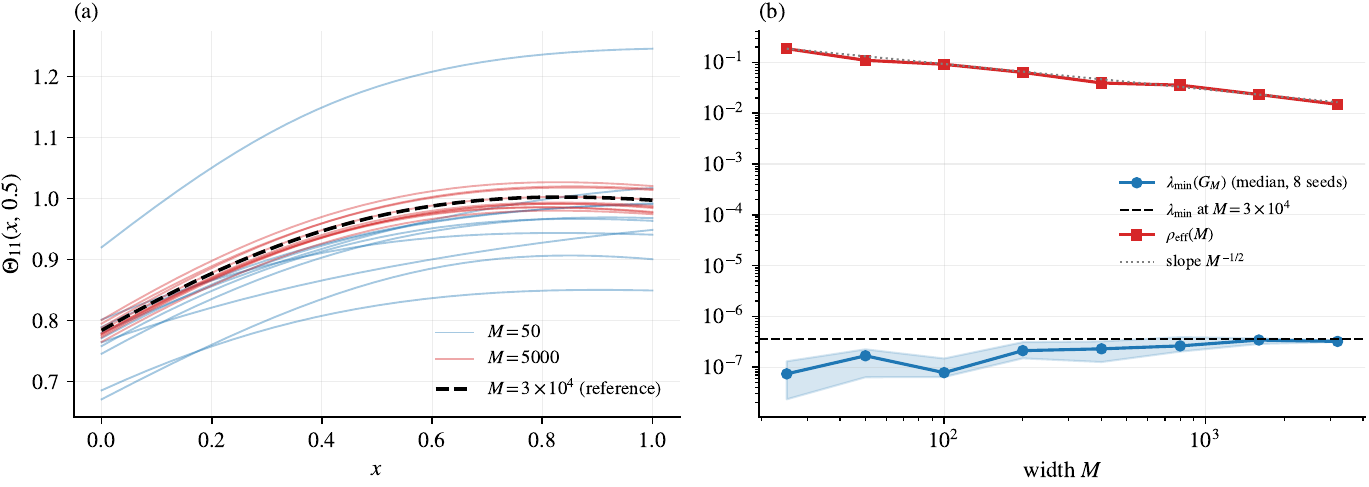}
		\caption{Finite width networks. Left, slices of the empirical tangent
			kernel. Right, $\lmin(G_M)$ and correlation $\rho_{\mathrm{eff}}(M)$.}
		\label{fig:largeur}
	\end{figure}
	
	\begin{figure}[H]
		\centering
		\includegraphics[width=\linewidth]{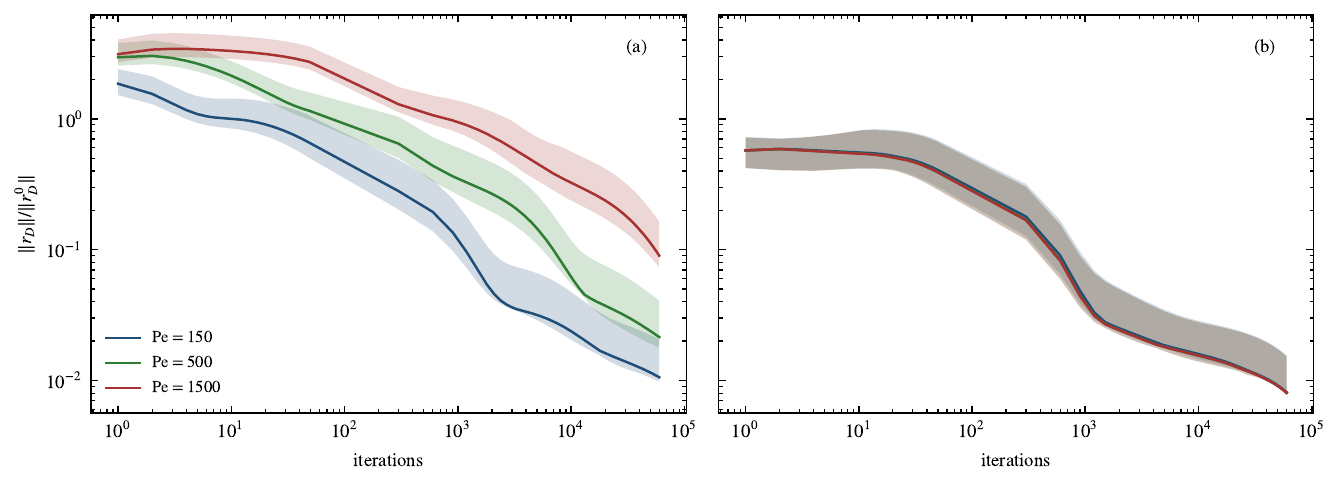}
		\caption{Full training, six initializations, median and interquartile
			band. (a) Without weighting. (b) With selective weighting, the three
			curves being superposed.}
		\label{fig:entrainement}
	\end{figure}
	
	\subsubsection{Solution error and stabilized reference}
	\label{ssec:num-l2}
	
	We train the preceding network on the manufactured solution of system
	\eqref{eq:sys3}, with a budget fixed at $1.5\times10^{4}$ iterations, and
	we report in Table~\ref{tab:l2} the relative error in $L^{2}$ norm of each
	component, computed as the median over three initializations, with seeds $0$, $1$ and $2$. Without
	weighting, the errors are unequally distributed among the components. The
	convected component is resolved to within a few percent for every Péclet
	number, while the errors of the diffusive components increase with
	$\Pe$. With weighting, the three components reach error levels between $7\%$ and $8\%$, with a much weaker dependence on $\Pe$.
	For the values of $\Pe$ located below the threshold, where the components are already balanced, the weighting increases the error by a factor close to $2$, from $3.74\%$ to $7.37\%$ at $\Pe=150$.
	
	For the same manufactured solution, we compute a reference by continuous $P_1$ finite elements on a uniform mesh, with a streamline-upwind Petrov--Galerkin stabilization applied to the convected component only. The discrete problem reads: find $u_h=(u_{h,1},u_{h,2},u_{h,3})$, with continuous piecewise linear components vanishing at $0$ and $1$, such that
\[
\sum_{i=1}^{3}\Bigl(\varepsilon_i\,\ip{u_{h,i}'}{v_i'}+\ip{b\,\delta_{i1}u_{h,i}'+(Hu_h)_i-f_i}{v_i+\delta_{i1}\tau_K\,b\,v_i'}\Bigr)=0,
\]
for all test functions $v=(v_1,v_2,v_3)$ in the same space. The integrals are computed elementwise by a six-point Gauss quadrature. On each element $K$ of size $h$, the stabilization parameter is
	\[
	\tau_K=\dfrac{h}{2b}\Bigl(\coth\mathrm{Pe}_K-\dfrac{1}{\mathrm{Pe}_K}\Bigr),
	\qquad
	\mathrm{Pe}_K=\dfrac{b\,h}{2\varepsilon_1}.
	\] This is the selective treatment used for
	systems of this nature in the finite element literature
	\citep{johnnovo2011,roos2008}. Table~\ref{tab:supg} reports the error as a
	function of the refinement. The scheme exhibits a convergence order close to
	two in $L^{2}$ norm, with an accuracy that remains uniform when $\Pe$
	varies. This reference shows that the manufactured solution can be approximated with an accuracy uniform in $\Pe$. At fixed network and budget, the weighting reduces the dependence of the training error on $\Pe$, which indicates a major contribution of the conditioning to the degradation observed without weighting.
	
	\begin{table}[H]
		\centering
		\caption{Relative $L^{2}$ error per component after $1.5\times10^{4}$
			iterations, median over three initializations. The last column gives
			the error of the SUPG reference on $64$ elements.}
		\label{tab:l2}
		\small
		\begin{tabular}{rlrrrrr}
			\toprule
			$\Pe$ & training & $u_1$ & $u_2$ & $u_3$ & total & SUPG\\
			\midrule
			$150$ & raw & $0.0326$ & $0.0402$ & $0.0412$ & $0.0374$ & $2.35\times10^{-4}$\\
			$150$ & weighted & $0.0767$ & $0.0681$ & $0.0796$ & $0.0737$ & \\
			$500$ & raw & $0.0215$ & $0.0422$ & $0.0395$ & $0.0353$ & \\
			$500$ & weighted & $0.0779$ & $0.0686$ & $0.0801$ & $0.0745$ & \\
			$1500$ & raw & $0.0153$ & $0.0805$ & $0.2091$ & $0.1005$ & $1.43\times10^{-4}$\\
			$1500$ & weighted & $0.0777$ & $0.0684$ & $0.0800$ & $0.0743$ & \\
			$5000$ & raw & $0.0178$ & $0.2942$ & $1.0488$ & $0.4642$ & $1.31\times10^{-4}$\\
			$5000$ & weighted & $0.0777$ & $0.0683$ & $0.0799$ & $0.0743$ & \\
			\bottomrule
		\end{tabular}
	\end{table}
	
	\begin{table}[H]
		\centering
		\caption{SUPG finite element reference. Relative $L^{2}$ error as a function of the
			number of elements and observed order.}
		\label{tab:supg}
		\small
		\begin{tabular}{rllllll}
			\toprule
			& \multicolumn{2}{c}{$\Pe=150$} & \multicolumn{2}{c}{$\Pe=1500$} & \multicolumn{2}{c}{$\Pe=5000$}\\
			\cmidrule(lr){2-3}\cmidrule(lr){4-5}\cmidrule(lr){6-7}
			$N$ & error & order & error & order & error & order\\
			\midrule
			$8$ & $1.12\times10^{-2}$ & & $9.45\times10^{-3}$ & & $9.32\times10^{-3}$ & \\
			$16$ & $2.97\times10^{-3}$ & $1.92$ & $2.18\times10^{-3}$ & $2.11$ & $2.13\times10^{-3}$ & $2.13$\\
			$32$ & $8.53\times10^{-4}$ & $1.80$ & $5.44\times10^{-4}$ & $2.01$ & $5.20\times10^{-4}$ & $2.04$\\
			$64$ & $2.35\times10^{-4}$ & $1.86$ & $1.43\times10^{-4}$ & $1.93$ & $1.31\times10^{-4}$ & $1.99$\\
			$128$ & $6.08\times10^{-5}$ & $1.95$ & $4.05\times10^{-5}$ & $1.82$ & $3.36\times10^{-5}$ & $1.96$\\
			\bottomrule
		\end{tabular}
	\end{table}
	
	\subsubsection{Adam optimizer and architecture}
	\label{ssec:num-adam}
	
	We compare the training by Adam of two architectures of equal total width:
	a shared trunk with three heads and three disjoint subnetworks. For
	each architecture, the learning rate is fixed once and for all,
	by pretuning at $\Pe=500$. We measure, over twenty initializations, the number
	of iterations needed to reach
	\[
	\norm{r_D}/\norm{r_D^{0}}<10^{-2},
	\]
	with a maximal budget of $30\,000$ iterations. The realizations that
	do not reach the tolerance within this budget are censored. They are
	counted as such and not replaced by an average value. The
	ratio between architectures then constitutes a lower bound. This ratio
	is computed initialization by initialization, the same seed being used
	for both architectures.
	
	Table~\ref{tab:adam} and Figure~\ref{fig:adam} gather the measurements. For system \eqref{eq:sys3}, formula \eqref{eq:seuil-ntk} gives $\Pe^{\star}(10)=292$. For the values of $\Pe$ located below this threshold, the two architectures
	require a number of iterations of the same order, with a modest advantage to the
	disjoint subnetworks. Beyond the threshold, the paired ratio equals $4.67$ at $\Pe=500$ and is at least $11.3$ at $\Pe=1500$, where five realizations out of twenty of the shared trunk reach the limit of $30\,000$ iterations. At $\Pe=5000$, twelve realizations out of twenty exhaust the budget, and the paired ratio is at least $21.9$. When a realization is censored, its number of iterations is counted at the budget, and the median of the paired ratios is then a lower bound. The median number
	of iterations of the disjoint subnetworks remains between $478$ and $1\,006$
	over the whole range considered.
	
	The protocol compares a trunk of width $M$ with three subnetworks of width
	$M/3$. The two architectures therefore have neither the same width per subnetwork,
	nor the same concentration regime of the tangent kernel. To isolate this effect,
	we also consider a comparison at identical width per
	subnetwork. The paired ratio then equals $6.29$, with an interval $[3.18,7.75]$, at $\Pe=500$, and is at least $5.84$ at $\Pe=1500$, against respectively $4.67$ and at least $11.3$ in the protocol at equal total width. The intervals overlap at the
	two Péclet numbers considered.
	
	The comparison with the star system of
	Section~\ref{ssec:num-pred1d} shows that the growth of the ratio beyond the
	threshold is not systematic. For \eqref{eq:sys3}, the ratio goes from
	$1.33$ to at least $21.9$, while, for the star system, it reaches
	$4.00$ in the neighborhood of the threshold and then falls back to $1.34$. The separation of the
	parameters is thus associated with a reduction of the training cost in the
	two systems, with a particularly visible effect in the neighborhood of the threshold.
	In contrast, the evolution of this gap when convection becomes very
	dominant depends on the structure of the system considered.
	
	\begin{table}[H]
		\centering
		\caption{Adam iterations over twenty initializations to reach
			$\norm{r_D}/\norm{r_D^{0}}<10^{-2}$, system \eqref{eq:sys3}. Median and
			interquartile interval, paired ratio with $95\%$ interval, and number
			of realizations reaching the budget.}
		\label{tab:adam}
		\small
		\begin{tabular}{rllll}
			\toprule
			$\Pe$ & shared trunk & disjoint subnetworks & paired ratio & censoring\\
			\midrule
			$50$ & $888$ $[753,1\,087]$ & $619$ $[510,754]$ & $1.33$ $[1.08,1.94]$ & $0/20$\\
			$150$ & $791$ $[668,1\,205]$ & $588$ $[473,702]$ & $1.43$ $[1.18,1.87]$ & $0/20$\\
			$500$ & $2\,391$ $[1\,944,3\,086]$ & $478$ $[413,525]$ & $4.67$ $[4.40,5.94]$ & $0/20$\\
			$1500$ & $6\,234$ $[2\,624,28\,629]$ & $515$ $[464,652]$ & $\ge 11.3$ & $5/20$\\
			$5000$ & $30\,000$ $[5\,932,30\,000]$ & $1\,006$ $[778,1\,232]$ & $\ge 21.9$ & $12/20$\\
			\bottomrule
		\end{tabular}
	\end{table}
	
	\begin{figure}[H]
		\centering
		\includegraphics[width=\linewidth]{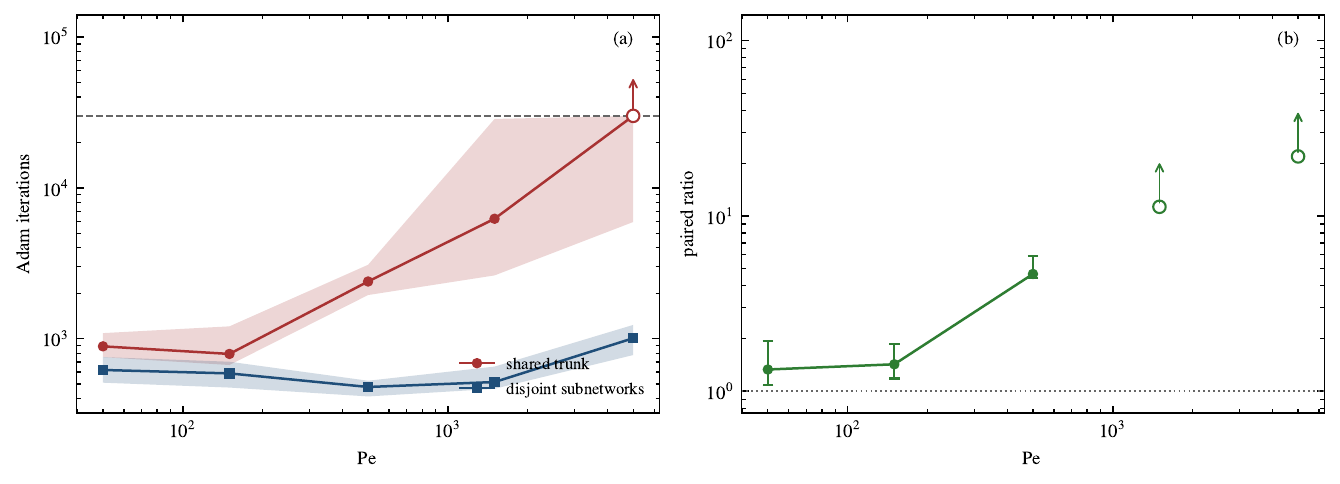}
		\caption{Training by Adam over twenty initializations. (a) Iterations to
			reach the tolerance, median and interquartile band. (b) Paired ratio
			between architectures. The hollow markers indicate a censored value.}
		\label{fig:adam}
	\end{figure}
	
	\section{Conclusion}
	\label{sec:conclusion}

We asked what slows down the training of a variational network on a heterogeneous coupled system under dominant convection, from which Péclet number, and how to mitigate it.

\emph{Mechanism (Q1).} In the neural tangent kernel regime, the Gram matrix undergoes a double screening. Its Schur complement relative to the convective block converges to
\[
S_{DD}\bigl(\Schur{\Kh}{\Kh_{CC}}\bigr)S_{DD}\he.
\]
The cross coupling blocks disappear from the limit, and the coupling gap identity quantifies what is lost. The exchanges contained in the diffusive block remain. The kernel enters only through its Schur complement. An exact identity gives the Schur complement at every finite convection strength. For a real symbol, the simple eigenvalues converge at order $b^{-2}$. The multiple eigenvalues, and all eigenvalues for a complex symbol, converge in general at order $b^{-1}$. The largest eigenvalue grows like $b^{2}$, so that the condition number grows like $\Pe^{2}$. For a fully connected network at infinite width, the kernel is diagonal and the second screening is a finite width effect. The canonical correlations are of order $M^{-1/2}$, and the reduction factor $1-\rho^{2}$ of Corollary~\ref{cor:rho} differs from one by $O(M^{-1})$.

\emph{Onset (Q2).} The threshold is a threshold of the conditioning. The two ends of the spectrum give a closed-form asymptotic approximation of the exact threshold, which itself is computed by a scalar root search. The frequency form \eqref{eq:seuil-precis} reproduces the exact threshold to better than $2.2\%$ over the four hundred and fifty-four configurations examined on the symbol, in dimensions one and two. These computations are spectral and use prescribed symbols. At infinite width, the threshold depends only on the symbol. The discrete form \eqref{eq:seuil-discret} rests on a discrete double screening theorem and needs two assemblies of the Jacobian. Under the non-degeneracy condition \eqref{eq:non-deg}, it reproduces the exact threshold of an actually constructed two-dimensional network to within $1\%$, on eight initializations of one architecture. Without the second projection, the median ratio is $1.075$. The training cost under Adam grows progressively across the threshold, without a sharp transition. The quantity $\cond_0$ must be computed for the operator whose departure is measured.

\emph{Remedies (Q3).} For the linearized descent with a fixed step size, the slowdown comes from the step size shared between equations. The stabilized finite element reference shows that the solution can be approximated with an accuracy uniform in $\Pe$. The experiments show that the conditioning is a major contributor to the degradation. The selective weighting keeps the condition number bounded at the reference frequency, at the price of an error almost twice as large below the threshold, and it can degrade the modes almost orthogonal to the transport. Disjoint subnetworks reduce the training cost, most markedly near the threshold. Their advantage is not explained by the canonical correlations of the kernel. Proposition~\ref{prop:groupes} shows that they admit a group-wise rescaling of the step size, while a shared trunk must then be slowed down relative to the other parameters. That Adam exploits this rescaling remains a hypothesis consistent with the measurements.

A natural extension concerns the evolution of the tangent kernel during training, outside the tangent kernel regime.

\section*{Declaration of competing interest}
The authors declare that they have no known competing financial interests or personal relationships that could have appeared to influence the work reported in this paper.

\section*{Data availability}
The scripts and data that reproduce all the figures and tables are available in a public repository \citep{zenodo}.

\end{document}